\documentclass{amsart}

\usepackage{circuitikz}
\usepackage{mathrsfs}
\usepackage{amsthm}
\usepackage{amscd}
\usepackage{amssymb}
\usepackage{latexsym}
\usepackage[justification=justified,singlelinecheck=false]{caption}
\usepackage{graphicx}
\usepackage{stmaryrd}
\usepackage{xypic}
\xyoption{curve}
\usepackage{nicefrac}
\usepackage{wasysym}
\usepackage{enumitem}
\usepackage{setspace}
\usepackage{tikz}
\usetikzlibrary{arrows}
\usepackage{ifthen}
\usepackage{verbatim}
\usepackage{hyperref}
\usepackage[margin=1.5in]{geometry}

\newcommand{\N}{\mathbb{N}}

\newcommand{\R}{\mathbb{R}}

\newcommand{\Z}{\mathbb{Z}}

\newcommand{\backspace}{\!\!\!\!\!\!\!}

\newcommand{\rel}{\rightharpoonup}

\newcommand{\ra}{\rightarrow}
\newcommand{\rra}{\rightrightarrows}

\newcommand{\xra}{\xrightarrow}

\newcommand{\ira}{\hookrightarrow}
\DeclareMathOperator*{\colim}{colim}
\DeclareMathOperator*{\holim}{holim}

\newcommand{\im}{{\mathop{im}\;\!\!}}

\newcommand{\GENERIC}{\mathscr{C}}

\newcommand{\DERIVED}{\mathscr{D}}

\newcommand{\DIAGRAM}{\mathscr{G}}

\newcommand{\Sh}{\mathop{Sh}\;\!\!\!}

\newcommand{\directsum}{\bigoplus}

\newcommand{\ring}[1]{{\ifthenelse{1=#1}{R}{{\ifthenelse{2=#1}{S}{T}}}}}
\newcommand{\semiring}[1]{{\ifthenelse{0=#1}{\N}{{\ifthenelse{1=#1}{S}{T}}}}}
\newcommand{\semimodule}[1]{{\ifthenelse{1=#1}{M}{N}}}
\newcommand{\semigroup}[1]{{\ifthenelse{0=#1}{\Z^+}{{\ifthenelse{1=#1}{G}{H}}}}}

\newcommand{\grading}{{\bullet}}

\newcommand{\tensor}[1]{\otimes_{\semiring{#1}}}

\newcommand{\homology}[3]{\!\!\;_{#3}H^{#1}_{#2}}
\newcommand{\cohomology}[3]{\!\!\;_{#3}H_{#1}^{#2}}

\newcommand{\sd}{\mathop{sd}\;\!\!}

\newcommand{\orientation}[1]{\mathcal{O}_{#1}}

\newcommand{\constantsheaf}[1]{k_{#1}}

\newcommand{\sheaf}[1]{\ifthenelse{1=#1}{\mathcal{F}}{\ifthenelse{2=#1}{\mathcal{G}}{\ifthenelse{3=#1}{\mathcal{H}}{\ifthenelse{4=#1}{\mathcal{I}}{\mathcal{J}}}}}}
\newcommand{\cosheaf}[1]{\ifthenelse{1=#1}{\mathcal{A}}{\ifthenelse{2=#1}{\mathcal{B}}{\ifthenelse{3=#1}{\mathcal{C}}{\ifthenelse{4=#1}{\mathcal{D}}{\mathcal{E}}}}}}

\newcommand{\facerelation}{{\trianglelefteqslant}}
\newcommand{\pseudoprod}{\prod}
\newcommand{\edgeweights}{\omega}
\newcommand{\connectinghom}{\delta}
\newcommand{\flowvalues}[3]{\substack{\text{$#1$-values of finite,}\\\text{locally $S$-decomposable}\\\text{$#3$-flows}}}
\newcommand{\cutvalue}[2]{[#1]_{e,#2}} 
\newcommand{\indecomposable}{indecomposable}
\newcommand{\digraph}[1]{{\ifthenelse{1=#1}{X}{{\ifthenelse{2=#1}{Y}{Z}}}}}
\newcommand{\edges}[1]{E_{#1}}
\newcommand{\vertices}[1]{V_{#1}}
\newcommand{\ncdihomology}[1]{H_{#1}}
\newcommand{\dihomology}[1]{H^c_{#1}}
\newcommand{\ncdicohomology}[1]{H^{#1}}
\newcommand{\dicohomology}[1]{H^{#1}_c}

\newcommand{\chainsheaf}[1]{\mathcal{S}_{#1}}
\newcommand{\COEFFICIENTS}{{\mathscr{M}}_{S}}
\newcommand{\CONSTRAINTS}{\hat{\mathscr{M}}_S}
\newcommand{\PARTIALSETS}{\mathscr{P}}
\newcommand{\delplus}[1]{{\partial_{{\scalebox{#1}{$+$}}}\!}}
\newcommand{\delminus}[1]{{\partial_{{\scalebox{#1}{$-$}}}\!}}

\newtheorem*{prop:holim}{Proposition \ref{prop:holim}}
\newtheorem*{prop:flows}{Proposition \ref{prop:flows}}
\newtheorem*{prop:cut-values}{Proposition \ref{prop:cut-values}}
\newtheorem*{thm:pd}{Theorem \ref{thm:pd}}
\newtheorem*{thm:flatness}{Theorem \cite[2.2]{katsov2004flat}}
\newtheorem*{thm:frieze}{Theorem \cite[2.1]{frieze1984algebraic}}
\newtheorem*{prop:universal.coefficients}{Proposition \ref{prop:universal.coefficients}}
\newtheorem*{prop:exactness}{Proposition \ref{prop:exactness}}
\newtheorem*{thm:equalizer}{Theorem \ref{thm:equalizer}}
\newtheorem*{thm:mfmc}{Theorem \ref{thm:mfmc}}
\newtheorem*{cor:mfmc}{Corollary \ref{cor:mfmc}}
\newtheorem*{cor:naive.poset.mfmc}{Corollary \ref{cor:naive.poset.mfmc}}
\newtheorem*{cor:naive.lattice.mfmc}{Corollary \ref{cor:naive.lattice.mfmc}}
\newtheorem*{cor:classical.mfmc}{Corollary \ref{cor:classical.mfmc}}
\newtheorem*{eg:classical.mfmc}{Example \ref{eg:classical.mfmc}}

\newtheorem{thm}{Theorem}[section]
\newtheorem{cor}[thm]{Corollary}
\newtheorem{lem}[thm]{Lemma}
\newtheorem{prop}[thm]{Proposition}
\newtheorem*{lem*}{Lemma}

\theoremstyle{definition}
\newtheorem{defn}[thm]{Definition}
\newtheorem{eg}[thm]{Example}

\title{Flow-Cut dualities for sheaves on graphs}
\author{Sanjeevi Krishnan}

\begin{document}
\begin{abstract}
  This paper generalizes the Max-Flow Min-Cut (MFMC) theorem from the setting of numerical capacities to sheaves of partial semimodules over semirings on directed graphs.
  Motivating examples of partial semimodules include probability distributions, multicommodity capacity constraints, and logical propositions.
  Directed (co)homology theories $\dihomology{\grading},\dicohomology{\grading}$ for such sheaves describes familar constructs on networks.
  First homology $\dihomology{1}$ classifies locally decomposable flows, an orientation sheaf over a semiring generalizes directions, connecting maps from $\dihomology{1}$ to $\dihomology{0}$ assign values to flows, connecting maps from $\dicohomology{0}$ to $\dicohomology{1}$ assign values to cuts, and Poincar\'{e} Duality describes a decomposition of flows as local flows over cuts.
  A consequent interpretation of feasible flow-values as a homotopy limit generalizes MFMC for edge weights in certain ordered monoids [Frieze] and hence also classical MFMC.
  A failure for directed sheaf (co)homology to satisfy a natural generalization of exactness explains certain duality gaps.
\end{abstract}
\maketitle
\tableofcontents


\section{Introduction}
Sheaves encode local data.
Sheaf cohomology, by definition, classifies those global properties of local data invariant under equivalent representations of the same data. 
Sheaf cohomology of group-valued sheaves has seen recent applications in the inference of global properties of complex systems with known local structure \cite{ghrist2014eat}.
Some examples are upper bounds on bit-rates across coding networks \cite{ghrist2011applications}, minimum sampling rates for noisy signals \cite{robinson2013nyquist}, and race conditions on asynchronous microprocessors \cite{robinson2012asynchronous}.
However, the existence of inverses in groups ignores the irreversibility of states in dynamical systems.
For example, the (co)homology of a module-valued sheaf on an oriented simplicial complex is invariant under a change in orientations; properties of systems sensitive to directionality (like orientations) in their state spaces (like simplicial complexes) are undetectable by classical sheaf (co)homology. 

Maximum flow-values and minimum cut-values are examples of such properties on networks.
The maximum traffic speed in a transportation grid and the minimum cost of interrupting a supply chain are examples for \textit{routing networks}.
The maximum throughput of information that can be transmitted and the minimum bandwidth of distributed channels are examples for \textit{coding networks}.
On routing networks, maximum flow-values and minimum cut-values coincide (Algebraic MFMC) under restrictions on the decomposability of flows into loops \cite[Theorem 2.1]{frieze1984algebraic} or the acyclicity of the flows and weak cancellativity of the monoid of possible values \cite[Theorem 3.1]{frieze1984algebraic}.
On coding networks, maximum throughput and minimum bandwidth coincide provided the informaton is transmitted from a single source (MMFMC). 

Such flow-cut dualities on directed graphs resemble topological dualities. 
In fact, undirected graph cuts and undirected graph flows induce Poincar\'{e} dual cohomology and homology classes on an ambient compact surface \cite{chambers2012homology}.  
Moreover, solutions to distributed linear coding problems are elements in the zeroth cohomology of a \textit{network coding sheaf} \cite{ghrist2011applications}.
Flows resemble homology, cuts resemble cohomology, local capacities resemble a sheaf, and flow-cut dualities evoke the Poincar\'{e} duality
\begin{equation}
  \label{eqn:pd}
  \homology{}{p}{}(X;\sheaf{1})\cong\cohomology{}{n-p}{}(X;\orientation{}\otimes\sheaf{1}).
\end{equation}
between homology $\homology{}{p}{}(X;\sheaf{1})$ and cohomology $\cohomology{}{n-p}{}(X;\orientation{\ring{1}}\otimes\sheaf{1})$ up to local orientations $\orientation{}$ for $\sheaf{1}$ a sheaf of $\ring{1}$-modules over a weak homology $n$-manifold $X$ \cite[Theorem 3.2]{bredon1997sheaf} in the case $n=p=1$.

This paper formalizes that resemblance by generalizing the constructions in (\ref{eqn:pd}) for sheaves $\sheaf{1}$ on digraphs ($1$-dimensional directed spaces) that model network constraints of interest in applications.
Edge orientations generalize to \textit{orientation sheaves $\orientation{S}$ over semirings $S$ on digraphs}.
Local constraints on networks define \textit{cellular sheaves $\sheaf{1}$ of partial $\semiring{1}$-semimodules on digraphs}, \textit{partial $S$-sheaves} for short; in particular, edge weights form subsheaves of the constant sheaf at a semilattice ordered $S$-semimodule.
The comparison of values that the sheaf $\sheaf{1}$ takes at different cuts requires \textit{parallel transport} between the different stalks (local values) of $\sheaf{1}$.
\textit{Zeroth homology and first cohomology}, the classification of stalks modulo parallel transport, describe the possible values that flows and cuts can take.
The paper develops the following dictionary.

\vspace{.1in}
\begin{tabular}{r|l}
  {\bf classical} & {\bf generalization}\\
  capacity constraints & $S$-sheaves \\
  edge directions & $S$-orientation sheaf $\orientation{S}$ \\
  flows & $\sheaf{1}$-flows \\
  (locally) decomposable finite flows & $\dihomology{1}$ \\
  flow values & $\dihomology{0}$  \\
  cut values & $\dicohomology{1}$ \\
\end{tabular}
\vspace{.1in}

Thus generalizations of classical results in (co)homology theory translate into insights on network optimization.
For example, local criteria for when certain flat resolutions are not needed in the construction of directed homology [Theorem \ref{thm:equalizer}] translate into local criteria for when sheaf-valued flows locally decompose with respect to a ground semiring [Proposition \ref{prop:flows}].
Similarly, a limited version of the Universal Coefficients Theorem for Homology [Proposition \ref{prop:universal.coefficients}] translates into conditions under which a change of base semiring preserves local decomposability.
Decomposability, and hence local decomposability, over the natural semiring $\N$ is important in algebraic generalizations of the classical Ford-Fulkerson algorithm for computing maximal flows \cite{frieze1984algebraic}.
For another example, the following Poincar\'{e} Duality translates into the decomposability of certain flows into local flows over cuts.

\begin{thm:pd}
  Fix digraph $X$ and open $U\subset X$.
  There exist dotted arrows inside
  \begin{equation}
    \label{eqn:relationship}
  \xymatrix{
      **[l] \dicohomology{0}(X-U;\orientation{S}\tensor{1}\sheaf{1})
      \ar@{.>}[r]
      \ar@<.7ex>[d]^-{\connectinghom_-}\ar@<-.7ex>[d]_-{\connectinghom_+}
    & **[r] \dihomology{1}((X,U);\sheaf{1})
      \ar@<.7ex>[d]^-{\connectinghom_-}\ar@<-.7ex>[d]_-{\connectinghom_+}
    \\
      **[l] \dicohomology{1}((X,X-U);\orientation{S}\tensor{1}\sheaf{1})
      \ar@{.>}[r]
    & **[r] \dihomology{0}(U;\sheaf{1}), 
  }
  \end{equation}
  natural in partial $S$-sheaves $\sheaf{1}$ on $X$, making the diagram jointly commute.
  The top arrow is an isomorphism and the bottom arrow is a surjection.
  The bottom arrow is an isomorphism if $S$ is a ring and each vertex has positive total degree or each vertex in $X$ has both positive in-degree and positive out-degree.
\end{thm:pd}

A subsequent sheaf-theoretic generalization of MFMC characterizes the feasible values of finite, locally decomposable flows as the solution to an optimization problem minimizing generalized values $[C]$ of cuts $C$ with respect to a given sheaf and distinguished edge.

\begin{thm:mfmc}[Sheaf-Theoretic MFMC]
  The equality
  $$\flowvalues{e}{X}{\sheaf{1}}=\holim\;\!\!_C[C]\subset\bigcap_C\;[C],$$
  where $[C]=\cutvalue{C}{\orientation{S}\tensor{1}\sheaf{1}}$ and the homotopy limit is taken over all minimal $e$-cuts $C$ and the inclusion is an equality for the case $\dihomology{1}(-;\sheaf{1})$ exact at $e$, holds for the following data.
  \begin{enumerate}
    \item digraph $X$
    \item cellular sheaf $\sheaf{1}$ of $\semiring{1}$-semimodules on $X$
    \item edge $e$ in $X$ with $X-e$ acyclic
  \end{enumerate}
\end{thm:mfmc}

\textit{Duality gaps}, gaps between the maximum flow-value (homotopy limit) and the minimum cut-value (limit), disappear when the cohomology theory $\dicohomology{\grading}(-;\sheaf{1})$ satisfies a generalized exactness axiom.
An example highlights the failure of homologial exatness in a multicommodity duality gap [Example \ref{eg:gap}].
Local algebraic criteria on $\sheaf{1}$ for $\dihomology{\grading}(-;\sheaf{1})$ to be exact [Lemma \ref{lem:exactness}] translate into simple, MFMC-like observations.
The following corollary subsumes an Algebraic MFMC for acyclic flows taking values in a weakly cancellative $\delta$-monoid \cite[Theorem 3.1]{frieze1984algebraic} and hence an Algebraic MFMC for decomposable flows taking values in a general $\delta$-monoid \cite[Theorem 2.1]{frieze1984algebraic} after applying Proposition \ref{prop:universal.coefficients}.

\begin{cor:mfmc}[Algebraic MFMC]
  The equality
  $$\bigvee_{\substack{\text{finite and}\\\text{$S$-decomposable}\\\text{flow $\phi$}}}\backspace\backspace\phi(e)=\backspace\;\;\bigwedge_{\substack{\text{$e$-cut $C$}}}\sum_{c\in C}\omega_c.$$
  holds for the following data.
  \begin{enumerate}
    \item naturally complete inf-semilattice ordered $S$-semimodule $M$
    \item digraph $(X;\edgeweights)$ with edges weighted by elements in $M$
    \item edge $e$ in $X$ with $X-e$ acyclic
  \end{enumerate}
\end{cor:mfmc}

\section{Outline}
\addtocontents{toc}{\protect\setcounter{tocdepth}{1}}
The paper respectively defines and constructs examples of sheaves on digraphs, constructs and investigates (co)homology theories of such sheaves, and finally interprets such invariants as familar constructs from network theory. 

\subsection{Coefficients}
The coefficients of the directed (co)homology theories in this paper are \textit{partial $S$-sheaves}, cellular sheaves of partial $\semiring{1}$-semimodules on digraphs.
\textit{Semimodules} over \textit{semirings} generalize modules over rings by dropping the requirement that additive inverses exist. 
Semimodules, into which embed diverse algebraic varieties \cite{stronkowski2009embedding}, can encode numerical quantities (natural numbers under addition), stochastic quantities (distributions under convolution), or order-theoretic measurements (lattices under binary infima).
Optimization problems negatively encode constraints as additive ideals in semimodules.
\textit{Partial semimodules} over \textit{semirings}, not standard in the literature, generalize semimodules by dropping the requirement that addition and scalar multiplication be defined everywhere.
Partial semimodules are exactly the complements of additive ideals in semimodules [Proposition \ref{prop:partial.semimodules}, \ref{prop:partial.homomorphisms}].

Semimodules, much less partial semimodules, much less partial $S$-sheaves, lack many properties of Abelian categories typically used in constructions of (co)homology theories.
For example, semimodules over a general semiring do not contain enough injectives \cite{golan1992theory}.
For another example, exact sequences of semimodules (regarded as pointed sets) do not describe general equalizers and coequalizers of semimodules.
However, partial semimodules over a general semiring contain enough projectives, kernel-pairs of partial semimodules generalize short exact sequences, and tensor products of semimodules generalize tensor products of modules and extend to actions on partial semimodules [Proposition \ref{prop:tensors}]. 
This action is used to twist semimodules of (co)chains by coefficient sheaves of partial semimodules.

\subsection{(Co)homology}
Abelian sheaf (co)homology generalizes [Proposition \ref{prop:classical.cohomologie}] for directed algebraic topology \cite{fajstrup2006algebraic}.
The functors $\dicohomology{0},\dicohomology{1}$ equalize and coequalize coboundary operators from $0$-cochains to $1$-cochains.
The functors $\dihomology{1},\dihomology{0}$ equalize and coequalize dual boundary operators from $1$-chains to $0$-chains.
Under certain local algebraic or local geometric [Example \ref{eg:degree.bounds}] criteria, first directed homology coincides with a degree $1$ homology theory for higher categorical structures \cite{patchkoria2000homology}.

\begin{thm:equalizer}
  For each digraph $\digraph{1}$, there exists a dotted arrow in the diagram
  \begin{equation*}
    \xymatrix@C=2pc{
      \dihomology{1}(\digraph{1};\sheaf{1})
        \ar@{.>}[r]
    & \directsum_{e\in\edges{\digraph{1}}}\dicohomology{0}(\langle e\rangle;\sheaf{1})
        \ar@<.7ex>[rrrr]^-{\directsum_v\sum_{\delminus{.4} e=v}\dicohomology{0}(\delminus{.4} e\subset\langle e\rangle)}
        \ar@<-.7ex>[rrrr]_-{\directsum_v\sum_{\delplus{.4} e=v}\dicohomology{0}(\delplus{.4}e\subset\langle e\rangle)}
    & & &
    & \directsum_{v\in\vertices{\digraph{1}}}\sheaf{1}(v),
    }
  \end{equation*}
  natural in partial $S$-sheaves $\sheaf{1}$ on $X$, commute.
  Furthermore, the above diagram is an equalizer diagram if the ground semiring $S$ is a ring or for each vertex $v$ in $X$, $\sheaf{1}(v)$ is flat, the in-degree of $v$ is $1$, or the out-degree of $v$ is $1$.
\end{thm:equalizer}

In particular, $\dihomology{\grading}$ classifies directed loops on finite digraphs for $S=\N$ [Corollary \ref{cor:dicycles}], and coincides with a cosheaf homology for finite digraphs and ground rings \cite{bredon1997sheaf}.
Local homology $S$-semimodules $\dihomology{1}((X,X-x);\constantsheaf{S})$ define an \textit{orientation sheaf} $\orientation{S}$ over semirings $S$, generalizing classical orientation sheaves over rings.
Unlike orientation sheaves over rings on graphs, orientation sheaves over semirings on digraphs are generally not stalkwise free [Figure \ref{fig:freeness}].

\begin{prop:universal.coefficients}[Universal Coefficients]
  There exists an isomorphism
  \begin{equation*}
    \dihomology{1}(\digraph{1};\sheaf{1})\tensor{1}\semimodule{1}\cong\dihomology{1}(\digraph{1};\sheaf{1}\tensor{1}\constantsheaf{\semimodule{1}})
  \end{equation*}
  natural in partial $S$-sheaves $\sheaf{1}$ and flat $\semiring{1}$-semimodules $\semimodule{1}$.
\end{prop:universal.coefficients}

While cellular sheaves of (co)chains are not subdivision invariant, all four (co)homology functors $\dicohomology{0},\dicohomology{1},\dihomology{0},\dihomology{1}$ are subdivision invariant [Propositions \ref{prop:sd-cohomologie} and \ref{prop:sd-homologie}] and hence are really invariants of semimodule-valued sheaves on \textit{directed geometric realizations} \cite{fajstrup2006algebraic} of digraphs.

Connecting homomorphisms
$$\connectinghom_-,\connectinghom_+:\dihomology{1}((X,U);\sheaf{1})\ra \dihomology{0}(U;\sheaf{1}),\quad
  \connectinghom_-,\connectinghom_+:\dicohomology{0}(C;\sheaf{1})\ra \dicohomology{1}((X,C);\sheaf{1})$$
collectively generalize the connecting homomorphism for ordinary (co)homology [Propositions \ref{prop:Ab.connecting.maps} and \ref{prop:Ab.coconnecting.maps}].
Like the classical connecting homomorphism, $\connectinghom_-,\connectinghom_+$ fit into a generalization \cite{patchkoria2000homology} of a chain complex that is exact under certain conditions [Lemma \ref{lem:exactness}].
Unlike the classical connecting homomorphism, that generalized sequence generally is neither exact [Example \ref{eg:exactness}] nor even $1$-natural in pairs of digraphs [Example \ref{eg:non-cannonicity}].

\subsection{Networks}
\textit{Networks} refer to graphs equipped with extra structure, like directionality or capacity constraints.
Often that structure is encoded as edge weights with some implicit interpretation.
For example, the positive real weights of a coding network represent the maximum bit-rate for each channel.
For another example, the positive real weights of a multicommodity network represent an upper bound on the possible weighted sums of quantities [Example \ref{eg:multidimensional.constraints}].
Sometimes that structure includes possible operations at the vertices.
For example, the vertices of a logical circuit usually carry some simple logical gates [Example \ref{eg:circuits}].
Partial $S$-sheaves explicitly model such structures as local partial $S$-semimodules of allowable structures, like orientations, local messages, multiple quantities or logical states.

Unlike edge weights, the local values $\dicohomology{0}(C;\sheaf{1})$ that a general sheaf $\sheaf{1}$ takes at different regions $C$ of a network $X$ are not readily comparable to one another.
The images $\cutvalue{C}{\sheaf{1}}$ of homomorphisms $\dicohomology{0}(C;\sheaf{1})\ra\dicohomology{1}(X-e;\sheaf{1})$, the values of $\sheaf{1}$ modulo parallel transport up to orientation [Theorem \ref{thm:pd}], are readily comparable.
Intuitively, \textit{$e$-cuts} are those regions $C$ in a network where $\cutvalue{C}{\sheaf{1}}$ sufficiently samples global information measured in $[e]_{\sheaf{1}}$.
Formally, \textit{$e$-cuts} are the supports of $1$-cocycles in $\constantsheaf{S}$ cohomologous to a canonical $1$-cocycle supported at an edge $e$ [Lemma \ref{lem:cuts}]; in a certain sense, $e$-cuts in a digraph $X$ are (supports to representatives of the) Verdier dual to restriction
$$\ncdicohomology{0}(e\subset X;\constantsheaf{\Z}):\ncdicohomology{0}(X;\constantsheaf{\Z})\ra\ncdicohomology{0}(e;\constantsheaf{\Z})=\Z.$$

Classical real flows on a weighted digraph straightforwardly generalize to \textit{$\sheaf{1}$-flows}; capacity constraints are encoded by the sheaf $\sheaf{1}$ and flow-conservation generalizes to an equalizer condition.
Decomposability, finiteness, and acyclicity of flows also readily generalize.
In particular, an $\omega$-flow on a finite digraph $(X;\omega)$ weighted in a commutative monoid that decomposes into loops is precisely an $\omega$-flow that is locally $\N$-decomposable.

\begin{prop:flows}
  For each partial $S$-sheaf $\sheaf{1}$ on a digraph $X$, 
  $$\dihomology{1}(X;\sheaf{1})=\substack{\text{finite and}\\\text{locally decomposable}\\\text{$\sheaf{1}$-flows}}$$
  naturally.
  The above partial $S$-semimodule conatins all $\sheaf{1}$-flows for $X$ compact and either $S$ a ring or for each $v\in V_X$, the in-degree of $v$ is $1$, the out-degree of $v$ is $1$, or $\sheaf{1}(v)$ is flat.
\end{prop:flows}

The operation of assigning feasible values to semimodules of local flows is a colimit, the construction of flows from local flows is a limit, and colimits and limits do not generally commute.
\textit{Homotopy limits} [Proposition \ref{prop:holim}] circumvent these difficulties.
\textit{Duality gaps}, discrepencies between maximum flow-values and minimum cut-values in networks, can arise when the calculation of minima, a limit, does not amount to a homotopy limit.
The relationship between duality gaps and the failure of exactness in homology is highlighted in the multicommodity setting [Example \ref{eg:gap}].

\subsection{Conventions}
This paper ocassionally abuses notation and conflates an element $x$ in a set with $\{x\}$ and in particular sometimes lets $X-x$ denote $X-\{x\}$.
A diagram, some of whose arrows are written immediately stacked next to one another like in (\ref{eqn:relationship}), \textit{jointly commutes} if each of the two associated diagrams, obtained by removing all top arrows or all bottom arrows or all left arrows or all right arrows, commutes.

\addtocontents{toc}{\protect\setcounter{tocdepth}{2}}

\section{Coefficients}\label{sec:coefficients}
The coefficients of the (co)homology theories introduced in this paper are \textit{cellular sheaves} of \textit{partial $\semiring{1}$-semimodules} over \textit{digraphs}.

\subsection{Partial semimodules}\label{subsec:semimodules}
The category of commutative monoids and monoid homomorphisms is a closed monoidal category, whose closed structure sends a pair $(A,B)$ of commutative monoids to the commutative monoid of monoid homomorphisms $A\ra B$ with addition defined pointwise \cite{ci2002tensor}.
A \textit{semiring} is a monoid object in that category.
Concretely, a semiring is a set $S$ equipped with distinguished elements $0,1\in S$ and a pair $+_S,\times_S:S\times S\ra S$ of associative operations such that the following equations hold for all $x,y,z\in S$.
\begin{align*}
  x\times_S(y+_Sz) &= (x\times_Sy)+_S(x\times_Sz)\\
  0\times_Sx &= 0\\
  0+_Sx &= x\\
  x+_Sy &= y+_Sx\\
  1\times_Sx &= x\\
  x\times_S1 &= x
\end{align*}

Henceforth this paper takes all semirings to be commutative monoid objects; the multiplication $\times_S$ of a semiring $S$ is assumed to be commutative. 
Fix a semiring $S$.
An \textit{$S$-semimodule} is a module object over $S$.
Let $\COEFFICIENTS=\langle\COEFFICIENTS,\tensor{1},S\rangle$ denote the closed monoidal category of $S$-semimodules and $S$-homomomorphisms between them whose closed structure $\hom_S(M,N)$ sends a pair $(M,N)$ of $S$-semimodules to the $S$-semimodule of $S$-homomorphisms $M\ra N$ with addition and scalar multipliation defined pointwise.
An $\semiring{1}$-semimodule $\semimodule{1}$ is \textit{flat} if
$$-\tensor{1}\semimodule{1}:\COEFFICIENTS\ra\COEFFICIENTS$$
preserves equalizers.
Let $\semiring{1}[X]$ denote the \textit{free $S$-semimodule generated by a set $X$}, the coproduct in $\COEFFICIENTS$ of a family of copies of $S$ indexed by $X$; each $x\in X$ is identified with the generator $1$ of the $x$-indexed copy of $S$ in $S[X]$.
The category $\COEFFICIENTS$ is complete and cocomplete; moreover, filtered colimits commute with finite limits and finite products coincide with finite coproducts.

\begin{thm:flatness}
  \label{thm:flatness}
  The following are equivalent for an $\semiring{1}$-semimodule $\semimodule{1}$.
  \begin{enumerate}
    \item The $\semiring{1}$-semimodule $\semimodule{1}$ is flat.
    \item The $\semiring{1}$-semimodule $\semimodule{1}$ is a filtered colimit of free $\semiring{1}$-semimodules generated by finite sets.
  \end{enumerate}
\end{thm:flatness}

\begin{eg}
  \label{eg:flatness}
  Let $\Lambda$ denote the \textit{Boolean semiring}, the semiring $\{0,1\}$ such that
  $$x+_\Lambda y=\max(x,y),\quad x\times_\Lambda y=\max(x,y).$$
  A $\Lambda$-semimodule is a poset in which every finite set has a least upper bound; the monoid addition describes binary suprema.  
  The below Hasse diagrams describe $\Lambda$-semimodules, the left flat and the right not flat.
  
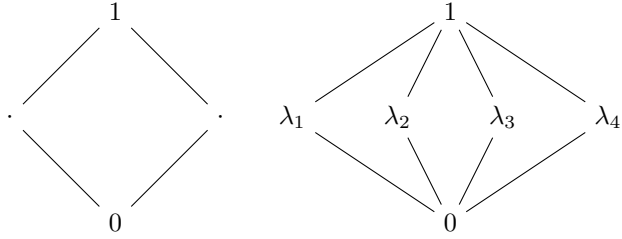
\begin{center}
  \begin{tikzpicture}[scale=.7]
  \node (one) at (0,2) {$1$};
  \node (a) at (-2,0) {$\cdot$};
  \node (d) at (2,0) {$\cdot$};
  \node (zero) at (0,-2) {$0$};
  \draw (zero) -- (a) -- (one) -- (d) -- (zero);
  \end{tikzpicture}  
  \hspace{.1in}
  \begin{tikzpicture}[scale=.7]
  \node (one) at (0,2) {$1$};
  \node (a) at (-3,0) {$\lambda_1$};
  \node (b) at (-1,0) {$\lambda_2$};
  \node (c) at (1,0) {$\lambda_3$};
  \node (d) at (3,0) {$\lambda_4$};
  \node (zero) at (0,-2) {$0$};
  \draw (zero) -- (a) -- (one) -- (b) -- (zero) -- (c) -- (one) -- (d) -- (zero);
  \end{tikzpicture}
  \end{center}
\end{eg}

\begin{eg}
  Over $\N$, $\N$ and $\R^{\geqslant 0}$ are flat while $\Z$ and $\R$ are not flat.
\end{eg}

Some of the algebraic structure of a semimodule is naturally described in terms of a natural preorder.
The \textit{natural preorder} on an $S$-semimodule $M$ is the preorder $\leqslant_M$ on the underlying set of $M$ whose relations are of the form
$$x\leqslant_M\lambda x+y,\quad x,y\in M,\;\lambda\in S-0.$$
An \textit{additive ideal} in an $S$-semimodule $M$ is a subset $I\subset M$ such that $(\lambda\times_Mx)+_My\in I$ for all $x\in M$, $y\in I$ and $0\neq\lambda\in S$.
In other words, additive ideals are the upper sets with respect to the natural preorder.
An $S$-semimodule $M$ is \textit{naturally complete} if it has all unique infima and unique suprema with respect to its natural preorder.
An $S$-semimodule $M$ is \textit{naturally inf-semilattice ordered} if every pair $x,y\in M$ admit a unique greatest lower bound $x\wedge y$ with respect to the natural preorder $\leqslant_M$ and $x\wedge (y+_Mz)=(x\wedge y)+_Mz$ for all $x,y,z\in M$.

A \textit{partial $S$-semimodule} will mean a set $M$ equipped with partial functions
$$+_M:M\times M\rel M,\quad\times_M:S\times M\rel M,$$
respectively called \textit{addition} and \textit{scalar multiplication}, from Cartesian products of underlying sets and distinguished element $0\in M$ such that in each of the following equations, the one side exists whenever the other side exists.
\begin{align}
  \label{eqn:axiom.first}0+_Mm &= m\\
  1\times_Mm &= m\\
  \label{eqn:axiom.associativity}(x+_My)+_Mz &= x+_M (y+_Mz)\\
  x+_My &= y+_Mx\\
  (\lambda_1+_S\lambda_2)\times_M x &= (\lambda_1\times_Mx)+_M(\lambda_2\times_Mx)\\
  (\lambda_1\times_S\lambda_2)\times_M x &= \lambda_1\times_M(\lambda_2\times_Mx)\\
  \lambda\times_M(x+_My) &= (\lambda\times_Mx)+_M(\lambda\times_My)\\
  0\times_Mm &= 0\\
  \label{eqn:axiom.last}\lambda\times_M0 &= 0
\end{align}

An $S$-semimodule is a partial $S$-semimodule whose addition and scalar multiplication define functions.
A \textit{partial $S$-subsemimodule} $A$ of a partial $S$-semimodule $B$ is an $S$-semimodule such that $A\subset B$ and addition and scalar multiplication on $A$ are restrictions and corestrictions of addition and scalar multiplication on $B$.
A \textit{partial $S$-semimodule} $A$ \textit{presents} an $S$-semimodule $B$ if $B$ is the quotient of the free $S$-semimodule $S[A]$ generated by the elements of $A$ modulo the smallest $S$-semimodule congruence $\equiv$ such that 
\begin{equation}
  \label{eqn:rewriting}
  (\lambda_1\times_{S[A]}a_1)+_{S[A]}(\lambda_2\times_{S[A]}a_2)\equiv(\lambda_1\times_{A}a_1)+_{A}(\lambda_2\times_{A}a_2)
\end{equation}
whenever the right side is defined.
Write $\langle M\rangle$ for the $S$-semimodule presented by a partial $S$-semimodule.

\begin{prop}
  \label{prop:partial.semimodules}
  The following are equivalent for a set $M$ equipped with partial functions
  $$+_M:M\times M\rel M,\quad\times_M:S\times M\rel M.$$
  \begin{enumerate}
    \item\label{item:partial.semimodules.1} $M$ is a partial $S$-semimodule.
    \item\label{item:partial.semimodules.2} $+_M,\times_M$ are restrictions and corestrictions of operations in an $S$-semimodule $N$ in which the complement of $M$ is an additive ideal in $N$.
  \end{enumerate}
  Given (\ref{item:partial.semimodules.1}), $N$ can be chosen in (\ref{item:partial.semimodules.2}) so that $M$ presents $N$.
\end{prop}
\begin{proof}
  Let $\lambda$ denote an element in $S$ and $m,m_1,m_2,m_3$ denote elements in $M$.

  Assume (\ref{item:partial.semimodules.1}).

  Let $F$ denote the free $S$-semimodule generated by the underlying set of $M$, $\equiv$ be the semimodule congruence on $F$ such that $M$ presents $F/\equiv$, and $N=F/\equiv$.
  The rules (\ref{eqn:axiom.first})-(\ref{eqn:axiom.last}) and (\ref{eqn:rewriting}) generate $\equiv$ and form a confluent and terminating rewriting system with normal forms including the elements of $M$.
  Hence $M$ is isomorphic to its image in $N$, a partial $S$-subsemimodule of $N$.
  Moreover $N-(M-0)$ is an additive ideal because in each of the rewriting rules (\ref{eqn:axiom.first})-(\ref{eqn:axiom.last}) and (\ref{eqn:rewriting}), the left-hand side does not contain summands in $F$ representing elements in $N-(M-0)$ when the right-hand side does not represent an element in $M$.
  Hence (\ref{item:partial.semimodules.2}).

  Assume (\ref{item:partial.semimodules.2}).

  Observe $0\in M$ by $N-(M-0)$ an $S$-semimodule.

  For all $m$, $0+_Nm=m$ and $0\times_Nm=0$.  
  Hence (\ref{eqn:axiom.first}) and (\ref{eqn:axiom.last}).

  Suppose $(m_1+_Mm_2)+_Mm_3$ exists.
  Then $m_1+_Mm_2$ exists.
  Then $m_2+_Mm_3$ exists because otherwise, $m_1+_N(m_2+_Nm_3)\notin M$ by $M-(N-0)$ an additive ideal in $N$, contradicting $m_1+_N(m_2+_Nm_3)=(m_1+_Nm_2)+_Nm_3=(m_1+_Mm_2)+_Mm_3\in M$.
  Hence $(m_1+_Mm_2)+_Mm_3=(m_1+_Nm_2)+_Nm_3=m_1+_N(m_2+_Nm_2)=m_1+_N(m_2+_Mm_3)=m_1+_M(m_2+_Mm_3)$.
  Hence (\ref{eqn:axiom.associativity}).

  The other conditional equations to check for (1) similarly follow.
\end{proof}

While partial $S$-semimodules are module objects over $S$, regarded as a monoid object in a monoidal category of sets and partial functions, \textit{partial $S$-homomorphisms} are more general than module homomorphisms between such module objects.
A (\textit{partial}) \textit{$S$-homomorphism} is a (partial) function $\psi:A\rel B$ from a partial $S$-semimodule $A$ to a partial $S$-semimodule $B$ such that $\psi(0)=0$ and the following holds whenever the left side is defined:
$$\psi\left(\left(\lambda_1\times_Ax_1\right)+_A\left(\lambda_2\times_Ax_2\right)\right)=\left(\lambda_1\times_B\psi(x_1)\right)+_B\left(\lambda_2\times_B\psi(x_2)\right)$$

\vspace{.1in}
Let $\CONSTRAINTS$ be the category of partial $S$-semimodules and partial $S$-homomorphisms.
The construction $\langle-\rangle$ naturally extends to a functor
$$\langle-\rangle:\CONSTRAINTS\ra\COEFFICIENTS.$$

The functor $\langle-\rangle$ acts like a left adjoint to the forgetful functor $\COEFFICIENTS\ra\CONSTRAINTS$.  
There exist isomorphisms $M\cong\langle M\rangle$ natural in $S$-semimodules $M$.
A \textit{partial identity} $\psi:A\rel B$ from a partial $S$-semimodule $A$ to a partial $S$-subsemimodule $B$ of $A$ is the partial function $A\rel B$ with $\psi(b)=b$ for all $b\in B$ and $\psi(a)$ undefined for all $a\in A-B$.
Partial identities are partial $S$-homomorphisms.
There exist inclusions $M\ira\langle M\rangle$ and partial identities $\langle M\rangle\rel M$ natural in partial $S$-semimodules $M$.
The existence of such natural partial $S$-homomorphisms can be used to show that $\CONSTRAINTS$ inherits various properties from $\COEFFICIENTS$, like having enough projectives.

\begin{prop}
  \label{prop:partial.homomorphisms}
  Consider partial $S$-semimodules $A$ and $B$.
  A partial function
  $$\psi:A\rel B$$
  is a partial $S$-homomorphism if and only if $\psi$ factors as the composite of a partial identity followed by an $S$-homomorphism between partial $S$-semimodules defined as the restriction and corestriction of an $S$-homomorphism between $S$-semimodules.
\end{prop}
\begin{proof}
  Let $\lambda_1,\lambda_2$ denote elements in $S$ and $a,a_1,a_2$ denote elements in $A$.

  Suppose $\psi$ is a partial $S$-homomorphism.
  Let $A_\psi$ denote the subset of $A$ consisting of all values on which $\psi$ is defined.
  For $x\in A-A_\psi-0$ and $a$, $a+x\in A-A_\psi$; otherwise $\psi(a+_Ax)=\psi(a)+_B\psi(x)$ and hence $\psi(x)$ would be defined, contradicting $x\notin A_\psi$.  
  Hence $A_\psi$ is a partial $S$-subsemimodule of $A$ [Proposition \ref{prop:partial.semimodules}], hence $\psi$ is the composite of the partial identity $A\rel A_\psi$ followed by a partial $S$-homomorphism $A_\psi\ra B$, which extends to an $S$-homomorphism $\langle A_\psi\rangle\ra\langle B\rangle$ [Proposition \ref{prop:partial.semimodules}].

  Conversely, suppose $\psi$ is the composite of a partial identity $S$-homomorphism followed by the restriction and corestriction of an $S$-homomorphism to a function of underlying sets.
  In order to show $\psi$ is a partial $S$-homomorphism, it suffices to consider the case $\psi$ a function of underlying sets defining the restriction and corestriction of an $S$-homomorphism 
  $$\psi':A'\ra B'.$$
  
  For all $\lambda_1,\lambda_2\in S$ and $a_1,a_2\in A$ with $(\lambda_1\times_Aa_1)+_A(\lambda_2\times_Aa_2)$ defined,
  \begin{align*}
     \psi\left(\left(\lambda_1\times_Aa_1\right)+_A\left(\lambda_2\times_Aa_2)\right)\right)
  &= \hat\psi\left(\left(\lambda_1\times_{A'}a_1\right)+_{A'}\left(\lambda_2\times_{A'}a_2\right)\right)\\
  &= \left(\lambda_1\times_{B}\psi'(a_1)\right)+_{B'}\left(\lambda_2\times_{B'}\psi'(a_2)\right)\\
  &= \left(\lambda_1\times_{B}\psi(a_1)\right)+_{B}\left(\lambda_2\times_{B}\psi(a_2)\right).
  \end{align*}
\end{proof}

\begin{cor}
  \label{cor:ab.categories}
  For each ring $S$, $\CONSTRAINTS=\COEFFICIENTS$ is the category of $S$-modules.
\end{cor}
\begin{proof}
  Every $S$-semimodule $M$ is an $S$-module because $0=(1-_R1)\times_Mx=x+_M(-1)\times_Mx$ for each $x\in M$ and hence each element in $M$ has an additive inverse.
  The additive ideals in an $R$-module $M$ are $0,M$ and hence every partial $S$-semimodule is an $R$-module [Proposition \ref{prop:partial.semimodules}].
  Moroever, every partial $R$-homomorphism of $R$-modules is an $R$-homomorphism [Proposition \ref{prop:partial.homomorphisms}].
\end{proof}

Coproducts in $\COEFFICIENTS$ extend to operations in $\CONSTRAINTS$ as follows.
The \textit{direct sum}, written
\begin{equation}
  \label{eqn:direct.sum}
  \directsum_{i\in\mathcal{I}}M_i,
\end{equation}
will refer to the partial $S$-semimodule, natural in an $\mathcal{I}$-indexed collection $\{M_i\}_{i\in\mathcal{I}}$ of partial $S$-semimodules, whose underlying set consists of those elements in the Cartesian product of underlying sets whose projections onto all but finitely many factors are $0$, and whose addition and scalar multiplication are defined coordinate-wise; each factor $M_i$ can be regarded as the partial $S$-subsemimodule of (\ref{eqn:direct.sum}) consisting of all terms whose projections onto all $j$-indexed factors are $0$ for $j\neq i$.

\begin{prop}
  \label{prop:co.equalizers}
  The category $\CONSTRAINTS$ is complete and cocomplete.  
  Inclusion
  $$\COEFFICIENTS\ira\CONSTRAINTS$$
  preserves equalizers, coequalizers, and filtered colimits. 
\end{prop}
\begin{proof}
  Let $\PARTIALSETS$ be the category of sets and partial functions.
  Let $U$ be the forgetful functor $\CONSTRAINTS\ra\PARTIALSETS$. 

  Consider a diagram $F:\DIAGRAM\ra\CONSTRAINTS$ in $\COEFFICIENTS$. 
  Let $d$ denote a $\DIAGRAM$-object.

  Let $L=\lim\;U\circ F$.
  Let $\pi_d$ denote a universal partial function $L\ra F(d)$.
  The operations $+_L:L\times L\rel L$ and $\times_L:S\times L\rel L$, defined by $x+_Ly=z$ if $\pi_dx+_L\pi_dy,\pi_dz$ are defined and coincide for each $\pi_d$ and $\lambda\times_Lx=y$ if $\pi_d\lambda\times_{F(d)}\pi_dx,\pi_dy$ are defined and coincide for each $d$, is terminal among all operations turning $L$ into a partial $S$-semimodule such that $\pi_d$ is a partial $S$-homomorphism for each $d$ and hence turns turns $L$ into a limit of $F$. 
  Equalizers in $\COEFFICIENTS$ are equalizers in $\CONSTRAINTS$ because equalizers in $\PARTIALSETS$ are equalizers in the category of sets and functions.
  
  Let $\iota_d$ denote a universal partial function $F(d)\ra\langle F(d)\rangle\ra M$.
  Let $Q$ be the smallest partial $S$-subsemimodule of $\colim\langle-\rangle\circ F$ containing the image of $\iota_d$ for each $d$.
  Then $\langle Q\rangle=\colim\langle-\rangle\circ F$ by universal properties of the colimit in $\COEFFICIENTS$.
  Consider a cocone from $F$ to a partial $S$-semimodule $T$.  
  In the diagram
  \begin{equation}
    \label{eqn:coequalizers}
    \xymatrix@C=1pc{
        \langle-\rangle\circ F
          \ar[rr]
          \ar[dr]
        & 
        & \langle Q\rangle
          \ar@{.>}[dl]
      \\
        & \langle T\rangle
      \\
        F
          \ar[rr]
          \ar[dr]
          \ar[uu]
        & 
        & Q
          \ar[uu]
          \ar@{.>}[dl]
      \\
        & T,
          \ar[uu]
    }  
  \end{equation} 
  where all vertical arrows are inclusions, there exists a unique dotted arrow making the upper triangle commute by universal properties of colimits in $\COEFFICIENTS$ and hence there exists a unique dotted arrow making the right parallegram, and hence lower triangle commute.  
  In the case $\DIAGRAM$ filtered or a pair of parallel arrows, the smallest partial $S$-subsemimodule of $\langle Q\rangle$ containing the image of $\iota_d$ for each $d$ is $\langle Q\rangle$.  
  Hence coqualizers and filtered colimits in $\CONSTRAINTS$ extend coequalizers and filtered colimits in $\COEFFICIENTS$.
\end{proof}

\begin{prop}
  \label{prop:tensors}
  The tensor product on $\COEFFICIENTS$ uniquely extends to a functor
  $$\tensor{1}:\COEFFICIENTS\times\CONSTRAINTS\ra\CONSTRAINTS,$$
  such that $-\tensor{1}B$ preserves $\directsum$ and coequalizers and $S\tensor{1}B\cong B$, all naturally in partial $S$-semimodules $B$.
  For each $S$-semimodule $A$, $A\tensor{1}-$ sends inclusions of the form $B\ira\langle B\rangle$ to inclusions of partial $S$-subsemimodules.
\end{prop}
\begin{proof}
  Uniqueness follows because every $S$-semimodule is the coequalizer in $\COEFFICIENTS$ and hence $\CONSTRAINTS$ [Proposition \ref{prop:co.equalizers}] of free $S$-semimodules.

  Let $A$ denote an $S$-semimodule and $B$ denote a partial $S$-semimodule.

  Define $S\tensor{1}B=B$, $(\alpha:S\ra S)\tensor{1}B$ to be the $S$-homomorphism $B\ra B$ sending $b\in B$ to $\alpha(1)\times_Bb$ for each $S$-homomorphism $\alpha:S\ra S$, and $S\tensor{1}\beta$ to be $\beta$ for each partial $S$-homomorphism $\beta$.
  Define $A\tensor{1}B$ to be the partial $S$-semimodule
  $$A\tensor{1}B=\bigoplus_{x\in X}S\tensor{1}B=\bigoplus_{x\in X}B$$
  natural in $A$ freely generated by a set $X$ and $B$.
  It suffices to show that the partial $S$-semimodule $A\tensor{1}B$ defined by the top coequalizer diagram
  \begin{equation*}
    \xymatrix@C=1pc{
        A_1\tensor{1}B
        \ar@<.7ex>[rrrrr]^-{\eta_-\tensor{1}B}
        \ar@<-.7ex>[rrrrr]_-{\eta_+\tensor{1}B}
        \ar[d]_{1_{A_1}\tensor{1}B\ira\langle B\rangle}
      &
      &
      &
      &
      & A_0\tensor{1}B
        \ar@{.>}[r]
        \ar[d]^{1_{A_0}\tensor{1}B\ira\langle B\rangle}
      & A\tensor{1}B
        \ar@{.>}[d]
      \\
        A_1\tensor{1}\langle B\rangle
        \ar@<.7ex>[rrrrr]^-{\eta_-\tensor{1}B}
        \ar@<-.7ex>[rrrrr]_-{\eta_+\tensor{1}B}
      &
      &
      &
      &
      & A_0\tensor{1}\langle B\rangle
        \ar[r]
      & A\tensor{1}\langle B\rangle
    }
  \end{equation*}
  natural in $B$, is independent of a choice of $S$-homomorphisms $\eta_-,\eta_+:A_1\ra A_0$ between freely generated $S$-semimodules $A_0,A_1$ with coequalizer $A$.
  For then $A\tensor{1}B$ is natural in $A$, $-\tensor{1}B$ preserves $\directsum$ and coequalizers by construction, and for $\eta_-,\eta_+:A_1\ra A_0$ a kernel-pair, the vertical arrows are inclusions and hence induce inclusions between coequalizers.

  Consider two such choices $\eta'_-,\eta'_+:A'_1\ra A'_0$ and $\eta''_-,\eta''_+:A''_1\ra A''_0$ of $S$-homomorphisms between freely generated $S$-semimodules with common coequalizer $A$.
  It suffices to show that $\colim\;(\eta'_-\tensor{1}B,\eta'_+\tensor{1}B)=\colim\;(\eta''_-\tensor{1}B,\eta''_+\tensor{1}B)$.
  Hence it suffices to consider the case $\eta'_-,\eta'_+$ and $\eta''_-,\eta''_+$ are kernel-pairs. 
  Consider the diagram
 \begin{equation*}
    \xymatrix@C=1pc{
        A'_1\tensor{1}B
        \ar@<.7ex>[rrr]^-{\eta'_-\tensor{1}B}
        \ar@<-.7ex>[rrr]_-{\eta'_+\tensor{1}B}
        \ar@{.>}[d]|-{\alpha_1\tensor{1}1_B}
      &
      &
      & A'_0\tensor{1}B
        \ar@{.>}[d]|-{\alpha_0\tensor{1}1_B}
        \ar@{.>}[r]
      & A\tensor{1}\langle B\rangle
        \ar@{=}[d]
      \\
        A''_1\tensor{1}B
        \ar@<.7ex>[rrr]^-{\eta''_-\tensor{1}B}
        \ar@<-.7ex>[rrr]_-{\eta''_+\tensor{1}B}
      &
      &
      & A''_0\tensor{1}B
        \ar@{.>}[r]
      & A\tensor{1}\langle B\rangle
    }
  \end{equation*}
  of solid arrows.
  There exist dotted horizontal dotted arrows induced by inclusion $B\ira\langle B\rangle$ and universal arrows $A'_0\ra A$ and $A''_0\ra A$.
  There exists an $S$-homomomorphism $\alpha_0:A'_0\ra A''_0$ making the right square commute by $A'_0$ projective.
  Hence there exists an $S$-homomomorphism $\alpha_1:A'_1\ra A''_1$ making the diagram jointly commutative.
  Hence the image the top horizontal dotted arrow lies in the image of the bottom horizontal dotted arrow.
  By symmetry the images of the horizontal dotted arrows coincide.
\end{proof}

A partial $\semiring{1}$-semimodule $\semimodule{1}$ is \textit{flat} if the functor
$$-\tensor{1}\semimodule{1}:\COEFFICIENTS\ra\CONSTRAINTS$$
preserves equalizers.
This definition of flatness for partial $S$-semimodules extends the definition of flatness for semimodules [Proposition \ref{prop:co.equalizers}].

\begin{lem}
  \label{lem:superflatness}
  For a flat $S$-semimodule $M$, $M\tensor{1}-:\CONSTRAINTS\ra\CONSTRAINTS$ preserves equalizers.
\end{lem}
\begin{proof}
  The operator $\bigoplus$ commutes with equalizers in $\CONSTRAINTS$ and hence the case $M$ free follows.
  The general case follows by $M$ a filtered colimit in $\COEFFICIENTS$, and hence $\CONSTRAINTS$ [Proposition \ref{prop:co.equalizers}], of finitely generated free $S$-semimodules [Theorem \ref{thm:flatness}].  
\end{proof}

\subsection{Digraphs}\label{subsec:digraphs}
A \textit{digraph} will mean a directed graph allowing for self-loops and missing vertices (like the first graph in Example \ref{eg:bifurcations}).
Formally, a \textit{digraph} $X$ means a set $\vertices{\digraph{1}}$ of \textit{vertices}, a set $\edges{\digraph{1}}$ of \textit{edges}, and \textit{source} and \textit{target} partial(ly defined) functions 
$$\delminus{.4},\delplus{.4}:\edges{\digraph{1}}\rel\vertices{\digraph{1}}.$$

This paper implicitly interprets a digraph $X$, with vertex and edge sets written as $V_X$ and $E_X$ as a poset $(X,\facerelation)$ and hence an Alexandroff space in the following sense.
The disjoint union $X=V_X\coprod E_X$ is partially ordered so that $v\facerelation e$ if $e\in\edges{X}$ and $v=\delminus{.4} e$ or $v=\delplus{.4}e$.
A subset $U\subset X$ is sometimes regarded as a digraph with source and target functions suitable restrictions and corestrictions of the original source and target partial functions to $U$.
A subset $C\subset X$ is \textit{open} if $v\facerelation e$ and $v\in C$ implies $e\in C$.
A subset $C\subset X$ is \textit{closed} if $v\facerelation e$ and $e\in C$ implies $v\in C$.
The \textit{closure} $\langle C\rangle$ of subset $C\subset X$ is the set
$$\langle C\rangle=C\cup\delplus{.4}(C\cap E_X)\cup\delminus{.4}(C\cap E_X).$$

A subset $C$ of a digraph $X$ is regarded itself as a digraph such that $E_C=E_X\cap C$, $V_C=V_C\cap C$, and the source and target partial functions of $C$ are restrictions and corestrictions of source and target partial functions of $X$.
A digraph is \textit{complete} if the above source and target partial functions are functions.
Like manifolds, digraphs have boundaries.  
Unlike manifolds, complete digraphs have both positive and negative boundaries defined as follows.  
For each complete digraph $X$, let
$$\delminus{.4}X=\delminus{.4}E_X-\delplus{.4}E_X,\quad \delplus{.4}X=\delplus{.4}E_X-\delminus{.4}E_X.$$

The \textit{in-degree} and \textit{out-degree} of a vertex $v\in\vertices{X}$ in a digraph $X$ are the cardinalities of the respective sets $\delminus{.4} ^{-1}(v),\delplus{.4}^{-1}(v)$.
A digraph is \textit{finite} if $V_X,E_X$ are finite, \textit{locally finite} if each vertex has finite in-degree and finite out-degree, and \textit{compact} if $\delminus{.4} ,\delplus{.4}:E_X\rel V_X$ are functions and $X$ is finite.
A \textit{directed loop} in a digraph $X$ is a compact subset of $X$ whose vertices each have in-degree and out-degree $1$; a directed loop is \textit{simple} if it contains no directed loops as proper subsets.
A digraph is \textit{directed acyclic} if it contains no directed loops.

The \textit{subdivision} $\sd X$ of a digraph $X$ is the digraph defined by
$$\vertices{\sd X}=X,\quad\edges{\sd X}=\{e_-\;|\;e\in\edges{X}\}\cup\{e_+\;|\;e\in\edges{X}\}$$
and $\delminus{.4}e_-=\delminus{.4}e$, $\delplus{.4}e_+=\delplus{.4}e$, and $\delminus{.4}e_+=\delplus{.4}e_-=e$ for each $e\in\edges{X}$.

\begin{eg}[Subdivisions]
  A digraph $X$ (left) and its subdivision $\sd X$ (right):
  
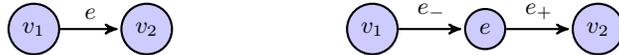
\label{fig:freeness}
  \begin{center}
  \begin{tikzpicture}[->,>=stealth',shorten >=1pt,auto,node distance=1.5cm,
                    thick,main node/.style={circle,fill=blue!20,draw,font=\sffamily\small}]
  \node[main node] (2) {$v_1$};
  \node[main node] (4) [right of=2] {$v_2$};
  \path[every node/.style={font=\small}]
    (2) edge node [sloped, pos=.5] {$e$} (4);
  \coordinate [right of=4] (46);
  \node[main node] (6) [right of =46] {$v_1$};
  \node[main node] (68) [right of =6] {$e$};
  \node[main node] (8) [right of =68] {$v_2$};
  \path[every node/.style={font=\small}]
    (6) edge node [sloped, pos=.5] {$e_-$} (68)
    (68) edge node [sloped, pos=.5] {$e_+$} (8);
  \end{tikzpicture} 
  \end{center} 
\end{eg}

A \textit{digraph} $(X;\omega)$ \textit{weighted in a commutative monoid $M$} will mean a digraph $X$ equipped with collection $\{w_e\}_{e\in E_X}$ of additive ideals $\omega_e$ in $M$ for all $e\in E_X$.
A \textit{weighted digraph} $(X;\omega)$ will simply refer to a digraph $(X;\omega)$ weighted in a given commutative monoid.

\subsection{Sheaves}\label{subsec:sheaves}
Edge weights, which are globally comparable as elements in a common set, generalize to \textit{sheaves}, a choice of locally varying sets over the vertices and edges, as follows.
Fix digraph $\digraph{1}$ and category $\GENERIC$.
Let $\Sh_{\digraph{1};\GENERIC}$ denote the category of functors
$$X\ra\GENERIC$$
from the poset $\digraph{1}$ and natural tranformations between them.
The \textit{restriction maps} of a $\Sh_{X;\GENERIC}$-object $\sheaf{1}$ \textit{between cells of $X$} are all $\GENERIC$-morphisms of the form $\sheaf{1}(v\facerelation e):\sheaf{1}(v)\ra\sheaf{1}(e)$ for all $v\facerelation e$ in $X$.
For each $C\subset X$, let $(C\subset X)^*$ denote the \textit{pullback} functor
$$(C\subset X)^*:\Sh_{X;\GENERIC}\ra\Sh_{C;\GENERIC}$$
defined on objects as restrictions.
Let $\sd$ denote the \textit{subdivision} functor
$$\sd:\Sh_{X;\GENERIC}\ra\Sh_{\sd X;\GENERIC}$$
naturally defined on objects $\sheaf{1}$ by the rules
\begin{align*}
  (\sd\sheaf{1})(c) &= \sheaf{1}(c),\;c\in X\subset V_{\sd X}\\
  (\sd\sheaf{1})(e_{\pm}) &= \sheaf{1}(e),\;e\in\edges{X}\subset V_{\sd X}\\
  (\sd\sheaf{1})(v\facerelation e_{\pm}) &= \sheaf{1}(v\facerelation e)\\
  (\sd\sheaf{1})(e\facerelation e_{\pm}) &= 1_{\sheaf{1}(e)}.
\end{align*}

An \textit{$S$-sheaf on $\digraph{1}$} will mean an $\Sh_{X;\COEFFICIENTS}$-object.
More generally, \textit{partial $S$-sheaf on $\digraph{1}$} will mean an $\Sh_{X;\CONSTRAINTS}$-object.
The \textit{constant sheaf at $\semimodule{1}$}, written $\constantsheaf{\semimodule{1}}$, is the $S$-sheaf on a digraph constant on a partial $\semiring{1}$-semimodule $\semimodule{1}$ as a functor.
Let
$$(C\subset X)_*:\Sh_{C;\CONSTRAINTS}\ra\Sh_{X;\CONSTRAINTS}$$
denote the right adjoint to $(C\subset X)^*$, the \textit{pushforward} functor naturally defined by
\begin{align*}
  (C\subset X)_*\sheaf{1}(c)&=\sheaf{1}(c), && c\in C\\
  (C\subset X)_*\sheaf{1}(c)&=0, && c\in X-C\\
  (C\subset X)_*\sheaf{1}(v\facerelation e)&=\sheaf{1}(v\facerelation e), && v\facerelation e,\;e,v\in C.
\end{align*}

A partial $S$-sheaf $\mathcal{A}$ on a digraph $X$ is a \textit{partial $S$-subsheaf} of an $S$-sheaf $\mathcal{B}$ on $X$ such that $\mathcal{A}(c)$ is a partial $S$-subsemimodule of $\mathcal{B}(c)$ for each $c\in C$ and objectwise inclusion defines a natural transformation $\mathcal{A}\ra\mathcal{B}$.

\begin{eg}[\'{E}tale Space]
  \label{eg:etale}
  Sheaves can be visualized as data over a graph as follows.
  
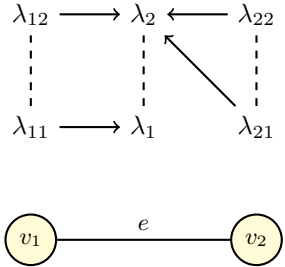
\vspace{.1in}
  \begin{center}
  \begin{tikzpicture}[auto,node distance=1.5cm,
                    thick,main node/.style={circle,fill=yellow!20,draw,font=\small}]

  \node[main node] (1) {$v_1$};
  \coordinate [right of =1] (12);
  \node[main node] (2) [right of=12] {$v_2$};

  \node[main node/.style={draw,font=\small}] (3) [above of=1] {$\lambda_{11}$};
  \node[main node/.style={draw,font=\small}] (4) [above of=3] {$\lambda_{12}$};
  \draw[dashed] (3) -- (4);

  \node[main node/.style={draw,font=\small}] (5) [above of=12] {$\lambda_{1}$};
  \node[main node/.style={draw,font=\small}] (6) [above of=5] {$\lambda_2$};
  \draw[dashed] (5) -- (6);

  \node[main node/.style={draw,font=\small}] (7) [above of=2] {$\lambda_{21}$};
  \node[main node/.style={draw,font=\small}] (8) [above of=7] {$\lambda_{22}$};
  \draw[dashed] (7) -- (8);

  \path[every node/.style={font=\small}]
    (1) edge node [sloped] {$e$} (2);
  \path[->]
    (3) edge node [sloped] {} (5)
    (4) edge node [sloped] {} (6)
    (7) edge node [sloped] {} (6)
    (8) edge node [sloped] {} (6);
  \end{tikzpicture} 
  \vspace{.1in}
  \end{center} 
  Let $X$ be the bottom graph $\{v_1,e,v_2\}$.
  The dotted lines over each cell $c\in X$ represent the Hasse diagrams of semilattices $\sheaf{1}(c)$, semimodules over the Boolean semiring $\Lambda$.
  The arrows reflect the $S$-homomorphisms of the form $\sheaf{1}(v\facerelation e)$ for vertices $v$ and edges $e$.
\end{eg}

Consider a digraph $(X;\omega)$ weighted in an $S$-semimodule $M$.  
This paper henceforth identifies the edge weights $\omega$ with the partial $S$-subsheaf on $X$ of $\constantsheaf{\semimodule{1}}$ defined by the following rules, where $\leqslant_M$ denotes the natural preorder on $M$.
\begin{align*}
  \omega(e) &= \{x\in M\;|\;x\leqslant_M\omega_e\} && e\in E_X\\
  \omega(v) &= \semimodule{1} && v\in V_X,
\end{align*}

The action of $\COEFFICIENTS$ on $\CONSTRAINTS$ defines an objectwise action
$$\tensor{1}:\Sh_{X;\COEFFICIENTS}\times\Sh_{X;\CONSTRAINTS}\ra\Sh_{X;\CONSTRAINTS}.$$

Likewise, the operation $\directsum$ on $\CONSTRAINTS$ defines an objectwise \textit{direct sum} operation on $\Sh_{X;\CONSTRAINTS}$.

A partial $S$-sheaf $\sheaf{1}$ is \textit{flat} if the following functor preserves equalizers:
$$-\tensor{1}\sheaf{1}:\Sh_{X;\COEFFICIENTS}\ra\Sh_{X;\CONSTRAINTS}$$
Equivalently, $\sheaf{1}$ is flat if it is objectwise flat.
The edge weights of a weighted digraph are flat if they take values in a flat $S$-semimodule.
A partial $S$-sheaf $\sheaf{1}$ is \textit{naturally inf-semilattice ordered} if it is objectwise naturally inf-semilattice ordered and the restriction maps between cells of $\sheaf{1}$ preserve greatest lower bounds, with respect to natural preorders, of finite subsets.

\addtocontents{toc}{\protect\setcounter{tocdepth}{2}}
\section{(Co)homology}\label{sec:H}
The constructions $\dicohomology{0},\dihomology{0}$ are defined as natural, dual categorical constructions on sheaves.  
In an abuse of notation, $\dicohomology{\grading}(-),\dihomology{\grading}(-)$ will denote $\dicohomology{\grading}(-;\sheaf{1}),\dihomology{\grading}(-;\sheaf{1})$ in this section whenever the coefficient sheaf $\sheaf{1}$ is understood.

\subsection{${\bf \dicohomology{\grading}}$}\label{subsec:HH}
Directed cohomology equalizes and coequalizes cochain diagrams.

\begin{defn}
  \label{defn:compact.HH0}
  Let $\ncdicohomology{\grading}(X;\sheaf{1})$ be defined the equalizer and coequalizer diagrams
  \begin{equation}
    \label{eqn:compact.HH0}
  \xymatrix@C=2pc{
      \ncdicohomology{0}(X;\sheaf{1})
        \ar@{.>}[r]
    & \directsum_{v\in V_X}\;\sheaf{1}(v)
        \ar@<.7ex>[rrr]^-{\directsum_{\delminus{.4} e=v}\sheaf{1}(v\facerelation e)}
        \ar@<-.7ex>[rrr]_-{\directsum_{\delplus{.4}e=v}\sheaf{1}(v\facerelation e)}
    &
    &
    & \directsum_{e\in E_X}\;\sheaf{1}(e)
      \ar@{.>}[r]
    & \ncdicohomology{1}(X;\sheaf{1})
  }
  \end{equation}
  natural in partial $S$-sheaves $\sheaf{1}$ on compact digraphs $X$.
\end{defn}

Inclusions $A\subset B\subset X$ of digraphs induce inclusions between direct sums of stalks and hence induce partial $S$-homomorphisms
$$\ncdicohomology{0}(A;(A\subset X)^*\sheaf{1})\ra\ncdicohomology{0}(B;(B\subset X)^*\sheaf{1})$$
natural in partial $S$-sheaves $\sheaf{1}$ on digraphs $X$.
Thus a compactly supported variant of directed cohomology can be defined as follows.

\begin{defn}
  \label{defn:compact.HH0}
  Let $\dicohomology{\grading}(X;\sheaf{1})$ denote the $S$-semimodule
  $$\dicohomology{\grading}(X;\sheaf{1})=\colim_{K\subset X}\ncdicohomology{\grading}(K;(K\subset X)^*\sheaf{1}),$$
  where $\ncdihomology{\grading}(K;\sheaf{1})$ is regarded as a covariant functor in $K$ and the colimit is taken over all compact subdigraphs $K\subset X$, natural in partial $S$-sheaves $\sheaf{1}$ on digraphs $X$.
\end{defn}

This paper abuses notation and henceforth denotes $\dicohomology{0}(C;(C\subset X)^*\sheaf{1})$ by $\dicohomology{0}(C;\sheaf{1})$ for $\sheaf{1}$ an $S$-sheaf on a digraph $X$.
Thus $\dicohomology{0}(C;\sheaf{1})$ is both covariant and contravariant in subsets $C$ of the digraph $X$ on which an $S$-sheaf $\sheaf{1}$ is defined.
For an inclusion $A\subset B$ of digraph, $\dicohomology{0}(A\subset B;\sheaf{1})$ will denote induced partial $S$-homomorphism
$$\dicohomology{0}(A\subset B;\sheaf{1}):\dicohomology{0}(B;\sheaf{1})\ra\dicohomology{0}(A;\sheaf{1}).$$

For $S$ a ring, $\CONSTRAINTS$ is the category of $S$-modules and the following proposition hence follows because the difference in the parallel arrows in Definition \ref{defn:compact.HH0} is the natural boundary operator on the Cech cochain complex for Abelian sheaf (co)homology.

\begin{prop}
  \label{prop:classical.cohomologie}
  Over a ring, $\dicohomology{\grading}(X;-)$ is compactly supported Abelian sheaf cohomology.
\end{prop}
\begin{proof}
  Take $S$ to be a ring.
  It suffices to consider the case $X$ compact [Proposition \ref{prop:co.equalizers}].
  In that case, the difference between the parallel arrows in (\ref{eqn:compact.HH0}) define the $0$th coboundary operator on the Cech complex associated to a sheaf $\sheaf{1}$ of $S$-modules, whose (co)kernel coincides with the (co)equalizer of (\ref{eqn:compact.HH0}) by $\CONSTRAINTS$ an Abelian category [Corollary \ref{cor:ab.categories}].
\end{proof}

\begin{prop}
  \label{prop:sd-cohomologie}
  There exist isomorphisms
  \begin{equation*}
    \label{eqn:sd-cohomologie}
    \dicohomology{\grading}(\sd X;\sd\sheaf{1})\cong\dicohomology{\grading}(X;\sheaf{1})
  \end{equation*}
  natural in digraphs $X$ and $S$-sheaves $\sheaf{1}$.  
\end{prop}
\begin{proof}
  Consider the natural jointly commutative diagram
  \begin{equation*}
  \xymatrix@C=2pc{
      \directsum_{x\in X}\;\sheaf{1}(x)
        \ar@<.7ex>[rrr]^-{\directsum_{\delminus{.4}e_\pm=x}\sheaf{1}(x\facerelation e_\pm)}
        \ar@<-.7ex>[rrr]_-{\directsum_{\delplus{.4}e_\pm=x}\sheaf{1}(x\facerelation e_\pm)}
        \ar[d]
    &
    &
    & \directsum_{e\in E_{X}}\;(\sd\sheaf{1})(e_-)\oplus(\sd\sheaf{1})(e_+),
        \ar[d]^{\sum_{e}(\sd\sheaf{1})(e_-)\cong\sheaf{1}(e))+(\sd\sheaf{1})(e_+)\sheaf{1}(e))}
    \\
      \directsum_{v\in V_X}\;\sheaf{1}(v)
        \ar@<.7ex>[rrr]^-{\directsum_{\delminus{.4} e=v}\sheaf{1}(v\facerelation e)}
        \ar@<-.7ex>[rrr]_-{\directsum_{\delplus{.4}e=v}\sheaf{1}(v\facerelation e)}
    &
    &
    & \directsum_{e\in E_X}\;\sheaf{1}(e)
  }
  \end{equation*}
  with the left vertical arrow projection.
  It suffices to show that the arrows $\dicohomology{\grading}(\sd X;\sd\sheaf{1})\ra\dicohomology{\grading}(X;\sheaf{1})$ between the (co)equalizers of the top and bottom rows of the diagram induced by the vertical arrows are isomorphisms.

  Let $\pi_x$ denote projection onto the $x$-indexed term.

  Every element $\sigma$ in the equalizer of the top row satisfies $\pi_e\sigma=\sheaf{1}(\delminus{.4}e\facerelation e)(\pi_{\delminus{.4}})$ for each $e\in E_X$ and hence the $\pi_e\sigma$'s for all $e\in E_X$ are determined by the $\pi_v\sigma$'s for all $v\in V_X$.
  Hence the induced arrow $\dicohomology{0}(\sd X;\sd\sheaf{1})\ra\dicohomology{0}(X;\sheaf{1})$ is injective.
  
  For each element $\sigma$ in the equalizer of the bottom row, $\sigma+\sum_{e\in E_X}\sheaf{1}(\delminus{.4}e\facerelation e)(\pi_{\delminus{.4}e}\sigma)$ lies in the equalizer of the top row and is the preimage of the left vertical arrow.
  Hence the induced arrow $\dicohomology{0}(\sd X;\sd\sheaf{1})\ra\dicohomology{0}(X;\sheaf{1})$ is surjective.

  The right vertical arrow is surjective. 
  Hence the induced arrow $\dicohomology{1}(\sd X;\sd\sheaf{1})\ra\dicohomology{1}(X;\sheaf{1})$ is surjective.

  Consider distinct elements $\sigma_1,\sigma_1$ in the top right term whose images $\bar\sigma_1,\bar\sigma_2$ under the right vertical map represent the same element in $\dicohomology{1}(X;\sheaf{1})$.
  It suffices to show $\sigma_1,\sigma_2$ represent the same element in $\dicohomology{1}(\sd X;\sd\sheaf{1})$.
  The quotient of the bottom right term by the smallest congruence $\equiv$ satisfying relations of the form $\sheaf{1}(v\facerelation e_1)(\sigma)\equiv\sheaf{1}(v\facerelation e_2)(\sigma)$ is the coequalizer of the bottom row.
  Hence it suffices to consider the case there exist $v\in V_X$, $e_1,e_2\in E_X$, and $\sigma\in\sheaf{1}(v)$ with $\sheaf{1}(v\facerelation e_i)(\sigma)=\bar\sigma_i$ for $i=1,2$; preimages of such $\bar\sigma_1,\bar\sigma_2$ under the right vertical arrow generate the top right term.
  Hence without loss of generality it suffices to consider the case $\sigma_1\in(\sd\sheaf{1})((e_1)_+)$ and $\sigma_2\in(\sd\sheaf{1})((e_2)_-)$ without loss of generality.
  In that case $(\sd\sheaf{1})(v\facerelation(e_1)_+)(\sigma)=\sigma_1$ and $(\sd\sheaf{1})(v\facerelation(e_2)_-)(\sigma)=\sigma_2$.
\end{proof}

The \textit{c-sections} of a partial $S$-sheaf are the elements in $\dicohomology{0}(U;\sheaf{1})$ for each open $U\subset X$.
Thus a \textit{c-sectionwise surjection} is a natural transformation $\epsilon$ of partial $S$-sheaves on a digraph $X$ such that $\dicohomology{0}(U;\epsilon)$ is a surjection for each open $U\subset X$.
Similarly, a \textit{c-sectionwise coequalizer diagram} is a diagram $\sheaf{1}_1\rra\sheaf{1}_0\ra\sheaf{1}$ of partial $S$-sheaves on a digraph $X$ such that $\dicohomology{0}(U;\epsilon)$ is a surjection for each open $U\subset X$

\begin{lem}
  \label{lem:cogodement}
  For each $S$-sheaf $\sheaf{1}$ on a digraph $X$, the universal natural transformation
  \begin{equation*}
    \epsilon:\directsum_{C\subset X}(C\subset X)_*\constantsheaf{\dicohomology{0}(C;\sheaf{1})}\ra\sheaf{1}
  \end{equation*}
  is a sectionwise surjection.
\end{lem}
\begin{proof}
  For each $B\subset X$, the partial $S$-homomorphism
  $$\dicohomology{0}(B;\epsilon):\directsum_{B\subset C}\dicohomology{0}(C;\sheaf{1})\ra\dicohomology{0}(B;\sheaf{1})$$
  is surjective because its restriction to the $B$-indexed summand is the identity.
\end{proof}

\begin{lem}
  \label{lem:HH0-universal.coefficients}
  There exists a partial $S$-homomorphism
  \begin{equation}
    \label{eqn:HH0-universal.coefficients}
    \dicohomology{\grading}(X;\constantsheaf{\semimodule{1}}\tensor{1}\sheaf{1})\ra\semimodule{1}\tensor{1}\dicohomology{\grading}(X;\sheaf{1}),
  \end{equation}
  natural in partial $S$-sheaves $\sheaf{1}$ on digraphs $X$ and $S$-semimodules $M$, that is an isomorphism for the case $\dicohomology{\grading}=\dicohomology{0}$ and $M$ flat or the case $\dicohomology{\grading}=\dicohomology{1}$.
\end{lem}
\begin{proof}
  It suffices to consider the case $X$ compact.
  In the jointly commutative square
  \begin{equation*}
  \xymatrix@C=2pc{
      M\tensor{1}\left(\directsum_{v\in V_X}\;\sheaf{1}(v)\right)
        \ar@<.7ex>[rrrr]^-{(\directsum_{\delminus{.4} e=v}\sheaf{1}(v\facerelation e))\tensor{1}1_M}
        \ar@<-.7ex>[rrrr]_-{(\directsum_{\delplus{.4}e=v}\sheaf{1}(v\facerelation e))\tensor{1}1_M}
        \ar[d]_{\cong}
    &
    &
    &
    & M\tensor{1}\left(\directsum_{e\in E_X}\;\sheaf{1}(c_2)\right)
        \ar[d]^{\cong}
    \\
    \directsum_{v\in V_X}\;M\tensor{1}\sheaf{1}(c)
        \ar@<.7ex>[rrrr]^-{\directsum_{\delminus{.4} e=v}1_M\tensor{1}\sheaf{1}(v\facerelation e)}
        \ar@<-.7ex>[rrrr]_-{\directsum_{\delplus{.4}e=v}1_M\tensor{1}\sheaf{1}(v\facerelation e)}
    &
    &
    &
    & \directsum_{e\in E_X}\;M\tensor{1}\sheaf{1}(c_2),
  }
  \end{equation*}
  the natural vertical arrows are isomorphisms because $M\tensor{1}-$ preserves $\directsum$ [Proposition \ref{prop:tensors}] and hence induce an isomorphism of (co)equalizers.
  The (co)equalizer of the bottom row is the left side of (\ref{eqn:HH0-universal.coefficients}).
  The coequalizer of the top row is the right side of (\ref{eqn:HH0-universal.coefficients}) for the case $\dicohomology{\grading}=\dicohomology{1}$ because $\semimodule{1}\tensor{1}-$ preserves coequalizers [Proposition \ref{prop:tensors}].
  The equalizer of the top row is the right side of (\ref{eqn:HH0-universal.coefficients}) for the case $\dicohomology{\grading}=\dicohomology{0}$ for $\semimodule{1}\tensor{1}-$ equalizer preserving, or equivalently, $\semimodule{1}$ flat [Lemma \ref{lem:superflatness}].
\end{proof}

Ordinary sheaf cohomology is exact.
Directed cohomology comes equipped with connecting homomorphisms from degree $0$ to degree $1$, although the natural analogue of exactness in the semimodule-theoretic setting fails in general.

\begin{defn}
  Fix closed $C\subset X$.
  Define partial $S$-homomorphisms
  $$\connectinghom_-,\connectinghom_+:\dicohomology{0}(C;\sheaf{1})\ra\dicohomology{1}((X,C),\sheaf{1})$$
  by each of the two possible commutative squares of the form
  \begin{equation*}
    \xymatrix@C=5pc{
        **[l]\dicohomology{0}(C;\sheaf{1})
        \ar[r]
        \ar[d]_-{\delta_\pm}
      & \directsum_{v\in V_C}\sheaf{1}(v)  
        \ar[r]
      & **[r]\directsum_{v\in V_{\partial_{\pm}C}}\sheaf{1}(v)
        \ar[d]^-{\directsum_{v}\sum_{e}\sheaf{1}(v\facerelation e)}
      \\
        **[l]\dicohomology{1}((X,C);\sheaf{1})
      & \directsum_{e\in X-C}\sheaf{1}(e)
        \ar[l]
      & **[r]\directsum_{\delplus{.4}e\in\partial_{\pm}C}\sheaf{1}(e),
        \ar[l]
    }
  \end{equation*}
  where the left horizontal arrows are universal arrows and the right horizontal arrows are induced by projection and inclusion.
\end{defn}

Sequences of directed (co)homology $S$-semimodules connected by such partial $S$-homomorphisms sometimes exhibit a natural generalization of exactness.

\begin{prop}
  \label{prop:Ab.coconnecting.maps}
  Let $S$ be a ring.
  For an $S$-sheaf $\sheaf{1}$ on $\digraph{1}$ and closed $C\subset X$,
  $$\connectinghom_+-\connectinghom_-:\dicohomology{0}(C;\sheaf{1})\ra\dicohomology{1}((X,C);\sheaf{1})$$
  is the ordinary connecting homomorphism for compactly supported Abelian sheaf cohomology.
\end{prop}
\begin{proof}
  It is straightforward to verify that the sequence
  \begin{equation*}
    \dicohomology{0}(X;\sheaf{1})\xra{\dicohomology{0}(C\subset X)}\dicohomology{0}(C;\sheaf{1})\xra{\connectinghom_+-\connectinghom_-}\dicohomology{1}((X,C);\sheaf{1})\xra{\dicohomology{1}(X-C\subset X)}\dicohomology{1}(X;\sheaf{1})
  \end{equation*}
  is exact at each of its terms, and hence the lemma holds, for $\sheaf{1}=\constantsheaf{S}$, because ordinary simplicial Abelian cohomology is characterized by the Eilenberg-Steenrod axioms and $\dicohomology{\grading}(X;-)$ is ordinary compactly supported sheaf cohomology.
  The lemma holds for general cellular sheaves $\sheaf{1}$ of $S$-modules by naturality.
\end{proof}

\subsection{${\bf \dihomology{\grading}}$}\label{subsec:H0}
Zeroth and first homology are coequalizers and equalizers of a diagram of chains.

\begin{defn}
  \label{defn:homologie}
  Let $\dihomology{\grading}(X;\sheaf{1})$ be defined by the equalizer and coequalizer diagrams
  \begin{equation*}
    \xymatrix@C=1pc{
        \dihomology{1}(X;\sheaf{1})
        \ar@{.>}[r]
      & \directsum_{e\in E_{\sd X}}\;\dicohomology{0}(\langle e\rangle;\sd\sheaf{1})
        \ar@<.7ex>[rrrrr]^-{\directsum_{e}\dicohomology{0}(\delminus{.4} e\subset\langle e\rangle;\sd\sheaf{1})}
        \ar@<-.7ex>[rrrrr]_-{\directsum_{e}\dicohomology{0}(\delminus{.4} e\subset\langle e\rangle;\sd\sheaf{1})}
      & & & & & \directsum_{v\in V_{\sd X}}\;(\sd\sheaf{1})(c)
        \ar@{.>}[r]
      & \dihomology{0}(X;\sheaf{1})
    }
  \end{equation*}
  natural in direct sums in $\Sh_{X;\CONSTRAINTS}$ of pushforwards of constant partial $S$-sheaves of projective partial $S$-semimodules.
\end{defn}

First directed sheaf homology $\dihomology{1}(X;\sheaf{1})$ corresponds to an existing homology theory on the \textit{cellular cosheaf} on $X$ defined by pulling back $\sd\sheaf{1}$ along closed cells, at least when the ground semiring is a ring \cite{bredon1997sheaf}.

\begin{eg}
  \label{eg:parallel.transport}
  For $X$ and $\sheaf{1}$ the digraph and $\Lambda$-sheaf on $X$ from Example \ref{eg:etale},
  $$\dihomology{0}(X;\sheaf{1})=\Big\{\{\lambda_{11},\lambda_1\},\{\lambda_{12},\lambda_2,\lambda_{22},\lambda_{21}\}\Big\}\cong\Lambda.$$
\end{eg}

The definition of homology extends to general sheaves by taking appropriate analogues of projective resolutions.

\begin{prop}
  \label{prop:H1}
  The partial $S$-semimodule $\dihomology{\grading}(\digraph{1};\sheaf{1})$ defined by the coequalizer diagram
  \begin{equation*}
    \xymatrix@C=2pc{
      \dihomology{\grading}(\digraph{1};\sheaf{1}_1)
        \ar@<.7ex>[rr]
        \ar@<-.7ex>[rr]
    & &
     \dihomology{\grading}(\digraph{1};\sheaf{1}_0)
        \ar@{.>}[rr]
    & & \dihomology{\grading}(\digraph{1};\sheaf{1}),
    }
  \end{equation*}
  is independent of a choice of c-sectionwise coequalizer diagram $\sheaf{1}_1\rra\sheaf{1}_0\ra\sheaf{1}$ of partial $S$-sheaves with $\sheaf{1}_0,\sheaf{1}_1$ pushforwards of constant partial $S$-sheaves of projective partial $S$-semimodules.
\end{prop}
\begin{proof}
  Let $H'_{\grading}(X;\sheaf{1})$ be defined by the equalizer and coequalizer diagrams
  \begin{equation*}
    \xymatrix@C=1pc{
        H'_1(X;\sheaf{1})
        \ar@{.>}[r]
      & \directsum_{e\in E_{\sd X}}\;\dicohomology{0}(\langle e\rangle;\sd\sheaf{1})
        \ar@<.7ex>[rrrrr]^-{\directsum_{e}\dicohomology{0}(\delminus{.4} e\subset\langle e\rangle;\sd\sheaf{1})}
        \ar@<-.7ex>[rrrrr]_-{\directsum_{e}\dicohomology{0}(\delminus{.4} e\subset\langle e\rangle;\sd\sheaf{1})}
      & & & & & \directsum_{v\in V_{\sd X}}\;(\sd\sheaf{1})(c)
        \ar@{.>}[r]
      & H'_0(X;\sheaf{1})
    }
  \end{equation*}
  natural in $\sheaf{1}$.
  For a pair of c-sectionwise surjections
  $$\epsilon':\sheaf{1}'_0\ra\sheaf{1},\quad\epsilon'':\sheaf{1}''_0\ra\sheaf{1}$$
  from pushforwards of constant partial $S$-sheaves of projective partial $S$-semimodules, there exists a natural transformation $\eta:\sheaf{1}'_0\ra\sheaf{1}''_0$ such that $\epsilon''\circ\eta=\epsilon'$ by projectivity and hence $\im\;H'_\grading(X;\epsilon')\subset\im\;H'_\grading(X;\epsilon'')$ and hence also $\im\;H'_\grading(X;\epsilon'')\subset\im\;H'_\grading(X;\epsilon')$ by symmetry.  
  Hence $\dihomology{\grading}(X;\sheaf{1})=\im\;H'_\grading(X;\epsilon')=\im\;H'_\grading(X;\epsilon'')$.
\end{proof}


Unlike relative first cohomology, relative first homology is not defined as an absolute homology theory on a subdigraph.  

\begin{defn}
  \label{defn:relative.H1}
  Let $\dihomology{1}((X,U);\sheaf{1})$ denote the partial $S$-semimodule
  $$\dihomology{1}((X,U);\sheaf{1})=\dihomology{1}(X;(X-U\subset X)_*\constantsheaf{\semiring{1}}\tensor{1}\sheaf{1})$$
  for each digraph $X$, open $U\subset X$, and partial $S$-sheaf $\sheaf{1}$ on $X$.
\end{defn}

\begin{prop}
  \label{prop:sd-homologie}
  There exist isomorphisms
  $$\dihomology{\grading}(X;\sheaf{1})\cong\dihomology{\grading}(\sd X;\sd\sheaf{1})$$
  natural in partial $S$-sheaves $\sheaf{1}$ on digraphs $X$.
\end{prop}

For the case where each vertex in $X$ have positive in-degree and out-degree, a proof follows from Theorem \ref{thm:pd}.
The general case follows from naturality.

Under certain local geometric or algebraic criteria, the construction of $\dihomology{1}$ requires neither that the digraph $X$ be subdivided nor that the coefficients $\sheaf{1}$ be replaced by a generalized resolution $\sheaf{1}_1\rra\sheaf{1}_0$.

\begin{thm}
  \label{thm:equalizer}
  Fix a digraph $\digraph{1}$.
  There exists a monic dotted arrow making the diagram
  \begin{equation*}
    \xymatrix@C=2pc{
      \dihomology{1}(\digraph{1};\sheaf{1})
        \ar@{.>}[r]
    & \directsum_{e\in\edges{\digraph{1}}}\dicohomology{0}(\langle e\rangle;\sheaf{1})
        \ar@<.7ex>[rrrr]^-{\directsum_v\sum_{\delminus{.4} e=v}\dicohomology{0}(\delminus{.4} e\subset\langle e\rangle)}
        \ar@<-.7ex>[rrrr]_-{\directsum_v\sum_{\delminus{.4} e=v}\dicohomology{0}(\delplus{.4}e\subset\langle e\rangle)}
    & & &
    & \directsum_{v\in\vertices{\digraph{1}}}\sheaf{1}(v),
    }
  \end{equation*}
  natural in partial $S$-sheaves $\sheaf{1}$ on $\digraph{1}$, commute.
  Furthermore, the above diagram is an equalizer diagram if $S$ is a ring or for each $v\in V_X$, $\sheaf{1}(v)$ is flat, the in-degree of $v$ is $1$, or the out-degree of $v$ is $1$.
\end{thm}

A proof is deferred until the end of \S\ref{subsec:duality}.

\begin{eg}[Degree bounds]
  \label{eg:degree.bounds}
  The geometric criteria of Theorem \ref{thm:equalizer} disallows bifurcations in two directions at once, but still allows for such digraphs as trees and hexagonal grids.
  The two leftmost digraphs, unlike the right square grid, satisfy the criteria.
  
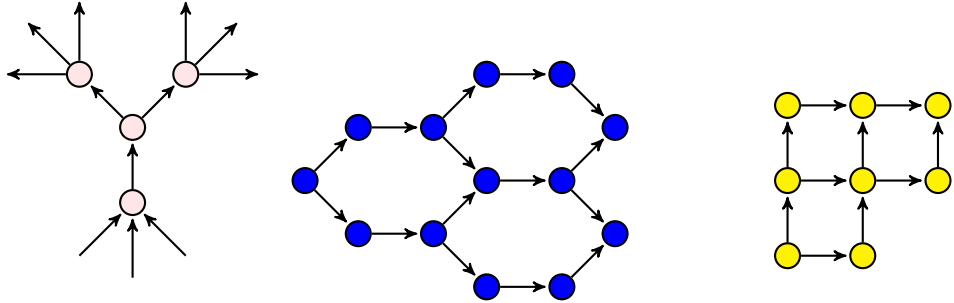
\begin{center}
  \begin{tikzpicture}[->,>=stealth',shorten >=1pt,auto,node distance=1cm,
                    thick,main node/.style={circle,fill=red!10,draw,font=\Large\bfseries}]
  \node[main node] (0) {};
  \node[main node] (1) [above right of=0] {};
  \coordinate [above of=1] (2);
  \coordinate [above right of=1] (3);
  \coordinate [right of=1] (4);
  \node[main node] (5) [above left of =0] {};
  \coordinate [above of =5] (6);
  \coordinate [above left of =5] (7);
  \coordinate [left of =5] (8);
  \node[main node] (9) [below of =0] {};
  \coordinate [below left of =9] (10);
  \coordinate [below of =9] (11);
  \coordinate [below right of =9] (12);
  \path[every node/.style={font=\small}]
    (0) edge node [right] {} (1)
    (0) edge node [right] {} (5)
    (9) edge node [right] {} (0)
    (1) edge node [right] {} (2)
    (1) edge node [right] {} (3)
    (1) edge node [right] {} (4)
    (5) edge node [right] {} (6)
    (5) edge node [right] {} (7)
    (5) edge node [right] {} (8)
    (10) edge node [right] {} (9)
    (11) edge node [right] {} (9)
    (12) edge node [right] {} (9);
  \coordinate [right of =0] (00);
  \coordinate [right of =00] (000);
  \coordinate [right of =000] (001);
  \node[main node,fill=blue] (13) [right of =001] {};
  \node[main node,fill=blue] (14) [above right of =13] {};
  \node[main node,fill=blue] (15) [right of =14] {};
  \node[main node,fill=blue] (16) [below right of =13] {};
  \node[main node,fill=blue] (17) [right of =16] {};
  \node[main node,fill=blue] (18) [below right of =15] {};
  \node[main node,fill=blue] (19) [below left of =16] {};
  \node[main node,fill=blue] (20) [below right of =19] {};
  \node[main node,fill=blue] (21) [right of =20] {};
  \node[main node,fill=blue] (22) [above right of =21] {};
  \node[main node,fill=blue] (23) [left of =13] {};
  \node[main node,fill=blue] (24) [below left of =23] {};
  \node[main node,fill=blue] (25) [below right of =24] {};
 \path[every node/.style={font=\small}]
    (13) edge node [right] {} (14)
    (14) edge node [right] {} (15)
    (13) edge node [right] {} (16)
    (16) edge node [right] {} (17)
    (17) edge node [right] {} (18)
    (15) edge node [right] {} (18)
    (19) edge node [right] {} (16)
    (19) edge node [right] {} (20)
    (20) edge node [right] {} (21)
    (21) edge node [right] {} (22)
    (17) edge node [right] {} (22)
    (23) edge node [right] {} (13)
    (24) edge node [right] {} (23)
    (24) edge node [right] {} (25)
    (25) edge node [right] {} (19);
  \coordinate [right of =17] (0000);
  \coordinate [right of =0000] (00000);
  \node[main node, fill=yellow] (26) [right of =00000] {};
  \node[main node, fill=yellow] (27) [right of =26] {};
  \node[main node, fill=yellow] (28) [above of =27] {};
  \node[main node, fill=yellow] (29) [above of =26] {};
  \node[main node, fill=yellow] (30) [right of =27] {};
  \node[main node, fill=yellow] (31) [right of =28] {};
  \node[main node, fill=yellow] (32) [below of =26] {};
  \node[main node, fill=yellow] (33) [right of =32] {};
 \path[every node/.style={font=\small}]
    (26) edge node [right] {} (29)
    (26) edge node [right] {} (27)
    (27) edge node [right] {} (28)
    (29) edge node [right] {} (28)
    (27) edge node [right] {} (30)
    (28) edge node [right] {} (31)
    (30) edge node [right] {} (31)
    (32) edge node [right] {} (26)
    (32) edge node [right] {} (33)
    (33) edge node [right] {} (27);
  \end{tikzpicture} 
  \end{center} 
\end{eg}

\begin{cor}
  \label{cor:dicycles}
  The number of simple directed loops in a digraph $X$ is
  $$\dim_{\R}\dihomology{1}(X;\constantsheaf{\N})\otimes_\N\R.$$
\end{cor}


\begin{prop}[Universal Coefficients]\label{prop:universal.coefficients}
  There exists an isomorphism
  \begin{equation*}
    \dihomology{1}(\digraph{1};\sheaf{1})\tensor{1}\semimodule{1}\cong\dihomology{1}(\digraph{1};\sheaf{1}\tensor{1}\constantsheaf{\semimodule{1}})
  \end{equation*}
  natural in partial $S$-sheaves $\sheaf{1}$ on digraphs $X$ and $\semiring{1}$-semimodules $\semimodule{1}$ that are flat on the left and right.
\end{prop}
\begin{proof}
  There exist natural isomorphisms
  \begin{equation*}
    \semimodule{1}\tensor{1}\dihomology{0}(\digraph{1};\orientation{S}\tensor{1}\tensor{1}\sheaf{1})\cong\dihomology{0}(\digraph{1};\orientation{S}\tensor{1}\constantsheaf{M}\tensor{1}\sheaf{1})
  \end{equation*}
  by Proposition \ref{prop:universal.coefficients}.
\end{proof}

\begin{eg}[Necessity of flatness]
  \label{eg:flatness.necessaire}
  Observe that
  $$\dihomology{1}(X;\constantsheaf{\N})\tensor{0}\Z=0\neq\dihomology{1}(X;\constantsheaf{\N}\tensor{0}\constantsheaf{\Z})=\dihomology{1}(X;\constantsheaf{\Z}).$$
  for $\digraph{1}$ a digraph with no dipaths infinite in both directions but at least one undirected path infinite in both directions.
  Hence tensoring with $\Z$, not flat as an $\N$-semimodule, fails to commute with $\dihomology{1}(X;-)$.
\end{eg}

Ordinary sheaf homology is exact.
Directed homology comes equipped with connecting homomorphisms from degree $1$ to degree $0$, although the natural analogue of exactness in the semimodule-theoretic setting fails in general.

\begin{defn}
  Fix open $U\subset X$.
  Let $\connectinghom_-,\connectinghom_+$ denote the $S$-homomorphisms
  $$\connectinghom_-,\connectinghom_+:\dihomology{1}((\digraph{1},U);\sheaf{1})\ra\dihomology{0}(U,\sheaf{1})$$
  defined by each of the two commutative diagrams of the form
  \begin{equation*}
    \xymatrix@C=5pc{
        \dihomology{1}(X,U)
        \ar[d]_{\connectinghom_{\pm}}
        \ar[r]
      & \bigoplus_{e\in E_{\sd(X-U)}}\dicohomology{0}(\langle e\rangle;\sd\sheaf{1})
        \ar[r]
      & \bigoplus_{v\in V_{\partial_\pm(X-U)}}(\sd\sheaf{1})(v)
        \ar[d]
      \\
        \dihomology{0}(U;\sheaf{1})
      & \bigoplus_{v\in V_{\sd U}}(\sd\sheaf{1})(v)
        \ar[l]
      & \bigoplus_{\partial_\mp e\in\partial_\pm(X-U)}(\sd\sheaf{1})(e)
        \ar[l]
    }
  \end{equation*}
  where the left horizontal arrows are universal arrows and the right horizontal arrows are induced by projection and inclusion.
\end{defn}

\begin{prop}
  \label{prop:Ab.connecting.maps}
  Let $S$ be a ring.
  For an $S$-sheaf $\sheaf{1}$ on $\digraph{1}$ and $C\subset\edges{\digraph{1}}$,
  $$\connectinghom_+-\connectinghom_-:\dihomology{1}((\digraph{1},C);\sheaf{1})\ra\dihomology{0}(C;\sheaf{1})$$
  is the ordinary connecting homomorphism for Borel-Moore sheaf homology.
\end{prop}

A proof will follow readily from Theorem \ref{thm:pd} for the case each vertex in $X$ has positive in-degree and positive out-degree and hence from naturality for the general case.

\begin{eg}[Failure of exactness]
  \label{eg:exactness}
  The commutative diagram
  \begin{equation*}
  \xymatrix@C=1pc{
      \dihomology{1}((\digraph{1},U);\sheaf{1})
        \ar@<.7ex>[rr]^-{\connectinghom_-}\ar@<-.7ex>[rr]_-{\connectinghom_+}
    & & \dihomology{0}(U,\sheaf{1})
        \ar[rrrr]^-{\dihomology{0}(U\subset\digraph{1})}
    & & & & \dihomology{0}(X;\sheaf{1})
  }
  \end{equation*}
  is not a coequalizer diagram for $\sheaf{1}=\constantsheaf{\N}$, $X$ the digraph below, and $U=\{v_1,e_1,e_2,v_2\}$.\\
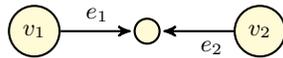
  \begin{center}
  \begin{tikzpicture}[->,>=stealth',shorten >=1pt,auto,node distance=1.5cm,
                    thick,main node/.style={circle,fill=yellow!20,draw,font=\small}]

  \node[main node] (1) {$v_1$};
  \node[main node] (2) [right of=1] {};
  \node[main node] (3) [right of=2] {$v_2$};

  \path[every node/.style={font=\small}]
    (1) edge node [sloped, pos=.5] {$e_1$} (2)
    (3) edge node [sloped, pos=.3] {$e_2$} (2);
  \end{tikzpicture} 
  \end{center} 
\end{eg}

\begin{defn}
  \label{defn:exactness}
  For each partial $S$-sheaf $\sheaf{1}$ on a digraph $X$, $\dihomology{\grading}(-;\sheaf{1})$ is \textit{exact at $U$} if
  \begin{equation*}
    \xymatrix@C=1pc{
      \dihomology{1}(X;\sheaf{1})
        \ar[rrrr]^-{\dicohomology{0}(X-U\subset X)}
    & & & & \dihomology{1}((X,U);\sheaf{1})
        \ar@<.7ex>[rr]^-{\connectinghom_-}\ar@<-.7ex>[rr]_-{\connectinghom_+}
    & & \dihomology{0}(U;\sheaf{1})
  }
  \end{equation*}
  commutes and the image of the left arrow coequalizes the parallel arrows, for each $C\subset X$.
\end{defn}

For $S$ a ground ring, $\dihomology{\grading}(-;\sheaf{1})$ is exact.  
The following lemma gives another criterion for exactness.

\begin{lem}
  \label{lem:exactness}
  For $S$-sheaf $\sheaf{1}$ on a digraph $X$ and open $U\subset X$,
  \begin{equation*}
  \xymatrix@C=1pc{
      \dihomology{1}(X;\sheaf{1})
        \ar[rrrr]^-{\dihomology{1}((X,\varnothing)\subset(X,U))}
    & & & & \dihomology{1}((X,U);\sheaf{1})
        \ar@<.7ex>[rr]^-{\connectinghom_-}\ar@<-.7ex>[rr]_-{\connectinghom_+}
    & & \dihomology{0}(U;\sheaf{1})
  }
  \end{equation*}
  commutes for $X-U,U$ acyclic and furthermore is exact if $\sheaf{1}$ is naturally inf-semilattice ordered and the images of restriction maps between cells of $\sheaf{1}$ are down-sets with respect to the natural preorders.
\end{lem}
\begin{proof}[proof sketch]
  It suffices to consider the case $X$ compact.
  The preorder $\leqslant_{U}$ on $U$ generated by the relations $\delminus{.4} e\leqslant_U\delplus{.4}e$ for each $e\in E_U$ makes $U$ a finite poset by acyclicity of $U$.
  The lemma follows by an inductive argument on the size of a maximal chain in $U$.
\end{proof}

\begin{eg}[Non-cannonicity of connecting maps]
  \label{eg:non-cannonicity}
  The diagram
  \begin{equation*}
  \xymatrix@C=2pc{
      \dihomology{1}((X,A);\sheaf{1})
        \ar[rrr]^-{\dihomology{1}((X,A)\subset(X,B);\sheaf{1})}
        \ar[d]_{\connectinghom}
    & & & \dihomology{1}((X,B);\sheaf{1})
        \ar[d]^{\connectinghom}
        \\
      \dihomology{0}(A;\sheaf{1})
        \ar[rrr]_-{\dihomology{0}(A\subset B;\sheaf{1})}
    & & & \dicohomology{0}(B;\sheaf{1}),
  }
  \end{equation*}
  need not commute for $\connectinghom=\connectinghom_-$ or $\connectinghom=\connectinghom_+$.
  For example, consider the case $X$ the digraph
  
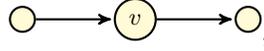
\begin{center}
  \begin{tikzpicture}[->,>=stealth',shorten >=1pt,auto,node distance=1.5cm,
                    thick,main node/.style={circle,fill=yellow!20,draw,font=\small}]

  \node[main node] (1) {};
  \node[main node] (2) [right of=1] {$v$};
  \node[main node] (3) [right of=2] {};

  \path[every node/.style={font=\small}]
    (1) edge node [sloped] {} (2)
    (2) edge node [sloped] {} (3);
  \end{tikzpicture},
  \end{center} 
  $A=X-v$, and $B=X$.
  Then $\dihomology{1}((X,A);\constantsheaf{\Z})=\Z$, $\dihomology{1}((X,B);\constantsheaf{\Z})=0$, $\dihomology{0}(A;\constantsheaf{\Z})=\Z\oplus\Z$, $\dihomology{0}(B;\constantsheaf{\Z})=\Z$, the left vertical map is an injection into one of the summands for $\partial=\delminus{.4} ,\delplus{.4}$, the bottom horizontal map is an isomorphism on each summand, but the right vertical map is the zero map.
\end{eg}

\begin{defn}
  \label{defn:orientations}
  Let $\orientation{S}$ denote the $S$-sheaf on a given digraph $X$ naturally defined by
  $$\orientation{S}(x)=\dihomology{1}((X,X-x);\constantsheaf{S}).$$
\end{defn}

The remainder of the section focuses on technical properties of orientation sheaves.
The sheaf $\orientation{S}$ is an \textit{orientation sheaf} of a weak homology 1-manifold \cite{bredon1997sheaf} in the following sense for $S$ a ring by the Universal Coefficients Theorem for homology.

\begin{lem}
  \label{lem:classical.orientations}
  Suppose $S$ is a ring.
  There exists an isomorphism
  $$\orientation{S}(c)\tensor{1}\semimodule{1}\cong\dihomology{1}((X,X-c);\semimodule{1})$$
  natural in cells $c$ in a given digraph and $S$-modules $\semimodule{1}$, where $\dihomology{\grading}$ denotes ordinary simplicial homology.
\end{lem}

The remainder of the section concerns technical observations of the orientation sheaf.  

\begin{lem}
  \label{lem:primal-orientations}
  Fix $v\in\vertices{\digraph{1}}$.
  The elements in
  $$(\delminus{.4} ^{-1}(v)\cap\delplus{.4}^{-1}(v))\cup\{e_-+e_+\;|\;e_-\in\partial^{-1}_-(v)-\partial^{-1}_+(v),\;e_+\in\partial^{-1}_+(v)-\partial^{-1}_-(v)\}.$$
  individually generate minimal $\N$-subsemimodules of $\orientation{\semiring{0}}$ and collectively generate all of $\orientation{\semiring{0}}(v)$.
\end{lem}
\begin{proof}
  Let $E_v,E^-_v,E^+_v$ be the sets
  $$E_v=\delminus{.4} ^{-1}v\cap\partial^{-1}_+v,\quad E^-_v=\partial^{-1}_-v- E_v,\quad E^+_v=\partial^{-1}_+v- E_v.$$

  Each $e\in E_v$, \indecomposable{} as an element in $\orientation{\semiring{0}}(v)$ by $e$ \indecomposable{} as an element in $\semiring{0}[E_G]$, lies in $\orientation{\semiring{0}}(v)$ because the parallel arrows both send $e$ to $1$.
  
  Consider $e_-\in E_v^-$ and $e_+\in E_v^+$.
  Then $e_-+e_+\in\orientation{\semiring{0}}(v)$ because both parallel arrows send $e_-+e_+$ to $1+0=0+1=1$.  
  Moreover, $e_-+e_+$ is \indecomposable{} because $e_-,e_+\notin\orientation{\semiring{0}}(v)$ by $e_-,e_+\notin E_v$.

  Every element in $\orientation{\semiring{0}}(v)$ factors as a sum of the form
  \begin{equation}
    \label{eqn:generating.orientations}
    \sum_{i\in\mathcal{I}}e_i+\sum_{i\in\mathcal{J}}e_j,\quad e_i\in E_v,\;i\in\mathcal{I}\quad e_j\in E^-_v\cup E^+_v,\;j\in\mathcal{J}.
  \end{equation}
  for some indexing sets $\mathcal{I},\mathcal{J}$.
  The first sum in (\ref{eqn:generating.orientations}) is generated by $E_v$.
  Moreover,
  $$\#\mathcal{I}+\#\{j\in\mathcal{J}\;|\;e_j\in E^-_v\}=\delminus{.4} (z)=\delplus{.4}(z)=\#\mathcal{I}+\#\{j\in\mathcal{J}\;|\;e_j\in E^+_v\},$$
  hence $\#\{j\in\mathcal{J}\;|\;e_j\in E^-_v\}=\#\{j\in\mathcal{J}\;|\;e_j\in E_v^+\}$, hence $\mathcal{J}$ is the disjoint union of bijective subsets $\mathcal{J}_-,\mathcal{J}_+$ such that $e_j\in E^-_v$ if $j\in\mathcal{J}_-$ and $e_j\in E^+_v$ if $j\in\mathcal{J}_+$.
  For any choice of bijection $\tau:\mathcal{J}_-\cong\mathcal{J}_+$, the second sum in (\ref{eqn:generating.orientations}) is generated by elements of the form $e_{j}+e_{\tau(j)}$ for $j\in\mathcal{I}_-$.
\end{proof}

\begin{lem}
  \label{lem:generating-orientations}
  Fix $v\in\vertices{\digraph{1}}$.
  Then 
  \begin{equation}
    \label{eqn:generating-orientations}
    (\delminus{.4} ^{-1}(v)\cap\delplus{.4}^{-1}(v))\cup\{e_-+e_+\;|\;e_-\in\partial^{-1}_-(v)-\partial^{-1}_+(v),\;e_+\in\partial^{-1}_+(v)-\partial^{-1}_-(v)\}.
  \end{equation}
  freely generates $\orientation{S}(v)$ if $v$ has in-degree or out-degree $1$.
\end{lem}
\begin{proof}
  Let $e_+$ denote an element in $\delplus{.4}^{-1}(v)$ and $e$ denote an element of the form $e_+$ or $e_-$.

  It suffices to consider the case $v$ has in-degree $1$, the case $v$ has out-degree $1$ symmetrically following.
  Then there exists a unique $e_-\in\delminus{.4} ^{-1}(v)$.
  The map $(\delminus{.4} )_v:\chainsheaf{1}(v)\ra\chainsheaf{0}(v)$ is the isomorphism $\semiring{1}[e_-]\cong\semiring{1}[v]$ sending $e_-$ to $v$.
  Hence
  \begin{equation*}
    \orientation{S}(v)=\left\{\sum_e\lambda_ee\;|\;\lambda_e\in\semiring{1},\;\lambda_{e_-}=\sum_{e_+}\lambda_{e_+}\right\}=\left\{\sum_{e_+}\lambda_{e_+}(e_-+e_+)\;|\;\lambda_{e_+}\in\semiring{1}\right\}=\semiring{1}[X]
  \end{equation*}
  for $X$ the set (\ref{eqn:generating-orientations}).
\end{proof}

\begin{lem}
  \label{lem:tensored.orientations}
  Fix $v\in\vertices{\digraph{1}}$.
  The diagram
  \begin{equation*}
    \xymatrix@C=3pc{
      \orientation{S}\tensor{1}\semimodule{1}\ar@{.>}[r]
    & \bigoplus_{e\in E_X}(\left\langle e\rangle\subset X)_*\constantsheaf{S}\right)\tensor{1}\semimodule{1}
      \ar@<.7ex>[r]^-{\delminus{.4} \tensor{1}\semimodule{1}}\ar@<-.7ex>[r]_-{\delplus{.4}\tensor{1}\semimodule{1}}
    & \bigoplus_{v\in V_X}\left(v\subset X)_*\constantsheaf{S}\right)\tensor{1}\semimodule{1},
    }
  \end{equation*}
  where the dotted arrow is induced by the natural inclusion $\orientation{S}\ra\chainsheaf{1}$, is an equalizer diagram natural in partial $\semiring{1}$-semimodules $\semimodule{1}$ if $\semimodule{1}$ is flat, $S$ is a ring, or $v$ has in-degree $1$, or $v$ has out-degree $1$.
\end{lem}
\begin{proof}
  For $\semimodule{1}$ flat, $\semimodule{1}\tensor{1}-$ preserves equalizer diagrams by definition.

  For $S$ a ring, the difference between parallel arrows in the above diagram is the degree $1$ differential in the chain complex of local simplicial chains at $v$ with coefficients in $\semimodule{1}$. 
  Hence the equalizer of the solid arrows in the statement of the lemma is the first local simplicial homology at $v$ with coefficients in $\semimodule{1}$ at $v$.
  That local homology module naturally is isomorphic to $\orientation{S}(v)\tensor{1}\semimodule{1}$ by the Universal Coefficients Theorem for Homology.

  Consider the case there exists a unique edge $e_-\in\edges{\digraph{1}}$ such that $\delminus{.4} e_-=v$.
  Let $e_+$ denote an element in $\delplus{.4}^{-1}(v)$.
  Then the diagram in the statement of the lemma is isomorphic to the diagram
  \begin{equation}
    \label{eqn:free.tensored.orientations}
    \xymatrix@C=4pc{
      \directsum_{e_+}\semimodule{1}\ar[r]^-{\directsum_{e_+}\iota_{e_+}}
   &  \directsum_{e\in\delminus{.4} ^{-1}(v)\cup\delplus{.4}^{-1}(v)}\semimodule{1}
      \ar@<.7ex>[r]^-{\delminus{.4} }\ar@<-.7ex>[r]_-{\delplus{.4}}
    & \semimodule{1},
    }
  \end{equation}
  by Lemma \ref{lem:generating-orientations}, where $\iota_{e_+}$ is the sum of inclusion of $\semimodule{1}$ into the $e_+$th summand and inclusion of $\semimodule{1}$ into the $e_-$th summand and $\partial$ maps the $e$th summand isomorphically onto $\semimodule{1}$ if $\delminus{.4} e=v$ and $0$ otherwise for $\partial=\delminus{.4} ,\delplus{.4}$.
  The diagram (\ref{eqn:free.tensored.orientations}) is an equalizer diagram by inspection.
\end{proof}

\begin{eg}[The freeness of orientations]
  Consider the digraphs
  
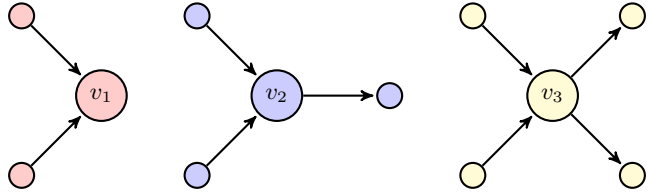
\vspace{.1in}
  \label{fig:freeness}
  \begin{center}
  \begin{tikzpicture}[->,>=stealth',shorten >=1pt,auto,node distance=1.5cm,
                    thick,main node/.style={circle,fill=red!20,draw,font=\sffamily\small}]

  \node[main node] (1) {};
  \node[main node] (2) [below right of=1] {$v_1$};
  \node[main node] (3) [below left of=2] {};

  \path[every node/.style={font=\small}]
    (1) edge node [right] {} (2)
    (3) edge node [right] {} (2);
  \end{tikzpicture} 
  \hspace{.2in}
  \begin{tikzpicture}[->,>=stealth',shorten >=1pt,auto,node distance=1.5cm,
                    thick,main node/.style={circle,fill=blue!20,draw,font=\sffamily\small}]

  \node[main node] (1) {};
  \node[main node] (2) [below right of=1] {$v_2$};
  \node[main node] (3) [below left of=2] {};
  \node[main node] (4) [right of=2] {};

  \path[every node/.style={font=\small}]
    (1) edge node [right] {} (2)
    (3) edge node [right] {} (2)
    (2) edge node [right] {} (4);
  \end{tikzpicture} 
  \hspace{.2in}
  \begin{tikzpicture}[->,>=stealth',shorten >=1pt,auto,node distance=1.5cm,
                    thick,main node/.style={circle,fill=yellow!20,draw,font=\sffamily\small}]

  \node[main node] (5) {};
  \node[main node] (6) [below right of=5] {$v_3$};
  \node[main node] (7) [below left of=6] {};
  \node[main node] (8) [above right of=6] {};
  \node[main node] (9) [below right of=6] {};

  \path[every node/.style={font=\small}]
    (5) edge node [right] {} (6)
    (7) edge node [right] {} (6)
    (6) edge node [right] {} (8)
    (6) edge node [right] {} (9);
  \end{tikzpicture} 
  \end{center} 
  While $\orientation{\N}(v_1)=0$ and $\orientation{\N}(v_2)\cong\N\oplus\N$ are free as $\N$-semimodules, $\orientation{\N}(v_3)$ is isomorphic to the quotient of $\N[\gamma_1,\gamma_2,\gamma_3,\gamma_4]$ modulo the relation $\gamma_1+\gamma_2=\gamma_3+\gamma_4$ and hence is not a free as an $\N$-semimodule. However, $\orientation{\N}(v_1)=\Z$, $\orientation{\Z}(v_2)\cong\Z\oplus\Z$ and $\orientation{\Z}(v_3)\cong\Z\oplus\Z\oplus\Z$ are all free as $\Z$-modules.
\end{eg}

Orientation sheaves on digraphs measure the degree to which a digraph bifurcates; in other words, orientation sheaves restrict to constant sheaves on cycles and directed paths unbounded in the past and future.

\begin{lem}
  \label{lem:constant.orientations}
  Consider a digraph $X$ such that one of the following holds:
  \begin{enumerate}
    \item Each vertex in $X$ has in-degree and out-degree both $1$.
    \item The semiring $S$ is a ring and each vertex in the digraph has total degree $2$.
  \end{enumerate}
  Then the orientation sheaf $\orientation{S}$ on $X$ is constant.
\end{lem}
\begin{proof}
  Consider the case that for each vertex $v$ there exist unique $e_-(v)\in\delminus{.4} ^{-1}(v)$ and $e_+(v)\in\delplus{.4}^{-1}(v)$.  
  Then $\orientation{S}(v)=S[e_-(v)+e_+(v)]$ and $\orientation{S}(\partial e\facerelation e)$ sends $e_-(\partial e)+e_+(\partial e)$ to $e_-(\partial e)$ or $e_+(\partial e)$ for $\partial=\delminus{.4} ,\delplus{.4}$ [Lemma \ref{lem:generating-orientations}].

  In the case $S$ is a ring and each vertex has total degree $2$, $\orientation{S}$ is the orientation sheaf over $S$ on a $1$-manifold, which is orientable over $S$.
\end{proof}

\begin{eg}[Constant orientations]
  \label{fig:constant.orientations}
  Consider the digraphs
  
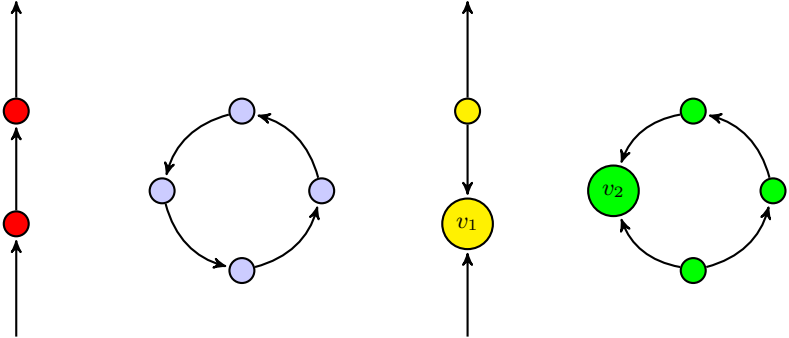
\vspace{.1in}
  \begin{center}
    \begin{tikzpicture}[->,>=stealth',shorten >=1pt,auto,node distance=1.5cm, thick,main node/.style={circle,fill=blue!20,draw,font=\sffamily\small}]
    \coordinate (1);
    \node[main node,fill=red] (2) [above of=1] {};
    \node[main node,fill=red] (3) [above of=2] {};
    \coordinate[above of=3] (4);
    \path[every node/.style={font=\sffamily\small}]
    (1) edge [right] node[left] {} (2)
    (2) edge [right] node[left] {} (3)
    (3) edge [right] node[left] {} (4);
    \coordinate[right of =3] (35);
    \node[main node] (5) [right of =35] {};
    \node[main node] (6) [below left of=5] {};
    \node[main node] (7) [below right of=6] {};
    \node[main node] (8) [below right of=5] {};
    \path[every node/.style={font=\sffamily\small}]
    (5) edge [bend right] node[left] {} (6)
    (6) edge [bend right] node[left] {} (7)
    (7) edge [bend right] node[right] {} (8)
    (8) edge [bend right] node[right] {} (5);
    \coordinate[right of =5] (510);
    \node[main node,fill=yellow] (11) [right of=510] {};
    \node[main node,fill=yellow] (10) [below of=11] {$v_1$};
    \coordinate[below of =10] (9);
    \coordinate[above of=11] (12);
    \path[every node/.style={font=\sffamily\small}]
    (9) edge [right] node[left] {} (10)
    (11) edge [right] node[left] {} (10)
    (11) edge [right] node[left] {} (12);
    \coordinate[right of =11] (1113);
    \node[main node,fill=green] (13) [right of=1113] {};
    \node[main node,fill=green] (14) [below left of=13] {$v_2$};
    \node[main node,fill=green] (15) [below right of=14] {};
    \node[main node,fill=green] (16) [below right of=13] {};
    \path[every node/.style={font=\sffamily\small}]
    (13) edge [bend right] node[left] {} (14)
    (15) edge [bend left] node[right] {} (14)
    (15) edge [bend right] node[right] {} (16)
    (16) edge [bend right] node[right] {} (13);
    \end{tikzpicture}
  \end{center} 
  Over the left two digraphs, $\orientation{\N}=\constantsheaf{\N}$.  
  Over all four digraphs, $\orientation{\Z}=\constantsheaf{\Z}$. 
  Over the right two digraphs, $\orientation{\N}(v_1)=\orientation{\N}(v_2)=0$ and hence $\orientation{\N}\neq\constantsheaf{\N}$.
\end{eg}

\begin{lem}
  \label{lem:sd-orientation}
  Fix a digraph $X$. 
  Over $\sd X$, $\sd\orientation{S}\cong\orientation{S}$ naturally.
\end{lem}

In the case $S$ a ring, the lemma is just the topological invariance of simplicial homology and a proof is just the observation that $\delplus{.4}-\delminus{.4} $ is subdivision invariant.
In the general case, $S$ may not even be embeddable in a ring and so $\orientation{S}$ cannot be readily constructed in an obviously invariant manner.

\begin{proof}[proof of Lemma \ref{lem:sd-orientation}]
  Define the following solid diagram of $S$-sheaves on $\sd X$ as follows.
  \begin{equation*}
    \xymatrix@C=3pc{
      \orientation{S}
      \ar[r]
      \ar@{.>}[d]
    &  \bigoplus_{e\in E_X}(\left\langle e\rangle\subset X)_*\constantsheaf{S}\right)
      \ar@<.7ex>[r]^-{\delminus{.4} }\ar@<-.7ex>[r]_-{\delplus{.4}}\ar[d]
    & \bigoplus_{v\in V_X}(v\subset X)_*\constantsheaf{S}
      \ar[d]
      \\
       \sd\orientation{S}
       \ar[r]
    &  \sd\bigoplus_{e\in E_X}(\left\langle e\rangle\subset X)_*\constantsheaf{S}\right)(v)
      \ar@<.7ex>[r]^-{\sd\delminus{.4} }\ar@<-.7ex>[r]_-{\sd\delplus{.4}}
    & \sd\bigoplus_{v\in V_X}(v\subset X)_*\constantsheaf{S},     
    }
  \end{equation*}
  Let the top left horizontal arrow be defined by the top equalizer diagram.
  Let the bottom left horizontal arrow be defined so that the bottom row is $\sd$ applied to an equalizer diagram, and hence also an equalizer diagram.
  Let the right vertical arrow be defined on each $v\in V_X\subset V_{\sd X}$ by the identity function and the $0$-map elsewhere.
  Let the middle arrow be defined on each $v\in V_X\subset V_{\sd X}$ and $e\in E_{\sd X}$ by the identity function and on each $e\in E_X\subset V_{\sd X}$ by the $S$-homorphism sending $e$ to $e_-+e_+$.
  The solid diagram jointly commutes by inspection and hence induces the dotted arrow.

  Over $v\in V_X\subset V_{\sd X}$, the solid and hence also dotted vertical arrows are isomorphisms.
  Over $e\in E_X\subset V_{\sd X}$, the dotted arrow is an isomorphism [Lemma \ref{lem:constant.orientations}].
  Over $e\in E_{\sd X}$, the top and bottom left horizontal arrows are isomorphisms, the middle vertical arrow is an isomorphism, and hence the dotted arrow is an isomorphism.
\end{proof}

\subsection{Duality}\label{subsec:duality}
The following duality relates cohomology with homology.

\begin{thm}
  \label{thm:pd}
  Fix digraph $X$ and open $U\subset X$.
  There exist dotted arrows inside
  \begin{equation}
    \label{eqn:relationship}
    \xymatrix{
      **[l] \dicohomology{0}(X-U;\orientation{S}\tensor{1}\sheaf{1})
      \ar@{.>}[r]
      \ar@<.7ex>[d]^-{\connectinghom_-}\ar@<-.7ex>[d]_-{\connectinghom_+}
    & **[r] \dihomology{1}((X,U);\sheaf{1})
      \ar@<.7ex>[d]^-{\connectinghom_-}\ar@<-.7ex>[d]_-{\connectinghom_+}
    \\
      **[l] \dicohomology{1}((X,X-U);\orientation{S}\tensor{1}\sheaf{1})
      \ar@{.>}[r]
    & **[r] \dihomology{0}(U;\sheaf{1}), 
  }
  \end{equation}
  natural in partial $S$-sheaves $\sheaf{1}$ on $X$, making the diagram jointly commute.
  The top arrow is an isomorphism and the bottom arrow is a surjection.
  The bottom arrow is an isomorphism if $S$ is a ring and each vertex has positive total degree or each vertex in $X$ has both positive in-degree and positive out-degree.
\end{thm}
\begin{proof}
  There exists a natural dotted monomorphism
  $$\Delta^1_{\sheaf{1}}:\dicohomology{0}(X-U;\orientation{S}\tensor{1}\sheaf{1})\cong\dicohomology{0}(\sd\;(X-U);\orientation{S}\tensor{1}\sd\sheaf{1})\cong\dihomology{1}((X,U);\sheaf{1})$$
  in (\ref{eqn:relationship}), defining an isomorphism for the case $\sheaf{1}$ flat [Lemma \ref{lem:pre.pd}] and hence surjective for the general case because objectwise projective partial $S$-sheaves are flat. 
  The diagram
  \begin{equation}
    \label{eqn:relationship}
    \xymatrix{
      **[l] \directsum_{v\in V_U}\orientation{S}(v)\tensor{1}\sheaf{1}(v)
      \ar[d]
      \ar@<.7ex>[r]^-{\connectinghom_-}\ar@<-.7ex>[r]_-{\connectinghom_+}
    & **[r] \directsum_{e\in E_U}\orientation{S}(e)\tensor{1}\sheaf{1}(e)
      \ar[d]
    \\
      **[l] \directsum_{e\in E_{\sd U}}\dicohomology{0}(\langle e\rangle;\sheaf{1})
      \ar@<.7ex>[r]^-{\connectinghom_-}\ar@<-.7ex>[r]_-{\connectinghom_+}
    & **[r] \directsum_{x\in V_{\sd U}}(\sd\sheaf{1})(x), 
  }
  \end{equation}
  where the left vertical arrow is induced by projections and the right vertical arrow is inclusion after making the identification $\orientation{S}(e)\tensor{1}\sheaf{1}(e)=\sheaf{1}(e)$ for $e\in E_X$, jointly commutes.
  Hence the vertical arrows induced an arrow $\Delta^0_{\sheaf{1}}:\dicohomology{0}(U;\orientation{S}\tensor{1}\sheaf{1})\ra\dihomology{0}(U;\orientation{S}\tensor{1}\sheaf{1})$.
  The arrow $\Delta^0_{\sheaf{1}}$ is surjective because every element in $\dihomology{0}(U;\sheaf{1})$ is represented by an element in $\directsum_{e\in E_U}(\sd\sheaf{1})(e)$.

  In the case $S$ is a ring and each vertex has positive total degree or each vertex has positive in-degree and positive out-degree, the left vertical arrow is surjective and hence $\Delta^0_{\sheaf{1}}$ is monic.

  The verification that (\ref{eqn:relationship}) jointly commutes is a straightforward diagram chase.
\end{proof}

\begin{cor}
  \label{cor:pd}
  For each $S$-sheaf over $\digraph{1}$,
  $$\dihomology{1}(\digraph{1};\sheaf{1})\cong\dicohomology{0}(\digraph{1};\sheaf{1})$$
  if each vertex has in-degree and out-degree both $1$, or $S$ is a ring and each vertex has total degree $2$.
\end{cor}
\begin{proof}
  Observe that
  $$\dihomology{1}(\digraph{1};\sheaf{1})=\dicohomology{0}(\digraph{1};\sheaf{1}\tensor{1}\orientation{S})\cong\dicohomology{0}(\digraph{1};\sheaf{1}\tensor{1}\constantsheaf{\semiring{1}})\cong\dicohomology{0}(\digraph{1};\sheaf{1}),$$
  the first equality by Theorem \ref{thm:pd}, the middle isomorphism by Lemma \ref{lem:constant.orientations}, and the last isomorphism by $\constantsheaf{\semiring{1}}$ a unit for $\otimes$ in $\semiring{1}h_{X;S}$.
\end{proof}

\begin{eg}[Necessity of restrictions]
  For $X$ the digraph in Example \ref{eg:exactness},
  $$\dihomology{0}(X;\constantsheaf{\N})\cong\N\ncong\N\oplus\N\cong \dicohomology{1}(X;\orientation{\N}).$$
\end{eg}

\begin{lem}
  \label{lem:pre.pd}
  Fix a compact digraph $\digraph{1}$. 
  There exist a monic dotted arrow in
  \begin{equation*}
    \xymatrix@C=2pc{
      \dicohomology{0}(\digraph{1})
        \ar@{.>}[r]
    & \directsum_{e\in\edges{\digraph{1}}}\dicohomology{0}(\langle e\rangle;\sheaf{1})
        \ar@<.7ex>[rrrr]^-{\directsum_v\sum_{\delminus{.4} e=v}\dicohomology{0}(\delminus{.4} e\subset\langle e\rangle)}
        \ar@<-.7ex>[rrrr]_-{\directsum_v\sum_{\delminus{.4} e=v}\dicohomology{0}(\delplus{.4}e\subset\langle e\rangle)}
    & & &
    & \directsum_{v\in\vertices{\digraph{1}}}\sheaf{1}(v)
    }
  \end{equation*}
  with $\dicohomology{\grading}(\digraph{1})=\dicohomology{\grading}(\digraph{1};\orientation{\semiring{1}}\tensor{1}\sheaf{1})$, natural in partial $S$-sheaves $\sheaf{1}$ on $\digraph{1}$ such that the diagram commutes.
  Furthermore, the left three terms in the diagram form an equalizer diagram if $\sheaf{1}$ is flat, $S$ is a ring, or each vertex in $\digraph{1}$ has in-degree $1$ or out-degree $1$.
\end{lem}
\begin{proof}
  There exists a universal natural transformation $\iota_S$ from $\orientation{S}$ making the diagram 
  \begin{equation*}
    \xymatrix@C=1pc{
        \bigoplus_{e\in E_{X}}(\langle e\rangle\subset X)_*\constantsheaf{S}\tensor{1}{\sheaf{1}}
          \ar@<.7ex>[rrrrr]^-{\bigoplus_{e}\constantsheaf{\dicohomology{0}(\delminus{.4} e\ira\langle e\rangle)}\tensor{1}1_{\sheaf{1}}}
          \ar@<-.7ex>[rrrrr]_-{\bigoplus_{e}\constantsheaf{\dicohomology{0}(\delplus{.4}e\ira\langle e\rangle)}\tensor{1}1_{\sheaf{1}}}
      & & & & & \bigoplus_{v\in V_{X}}(v\subset X)_*\constantsheaf{S}\tensor{1}{\sheaf{1}}.
    }
  \end{equation*}
  an equalizer diagram for $\sheaf{1}=\constantsheaf{S}$.
  Hence $\dicohomology{0}(X;\iota_S\tensor{1}1_{\sheaf{1}})$ induces a natural arrow to $\dihomology{1}(X;\sheaf{1})$.

  Consider the case $S$ a ring or for each $v\in V_X$, $\sheaf{1}$ is flat or $v$ has in-degree $1$ or $v$ has out-degree $1$.
  Then $\iota_S\tensor{1}\sheaf{1}$ equalizes the above diagram [Lemma \ref{lem:tensored.orientations}] and induces an arrow equalizing the diagram obtained by applying the equalizer-preserving functor $\dicohomology{0}(X;-)$.
\end{proof}

\begin{proof}[proof of Proposition \ref{prop:sd-homologie}]
  There exist natural isomorphisms
  \begin{equation*}
    \dicohomology{0}(X-C;\orientation{S}\tensor{1}\sheaf{1})
    \cong\dicohomology{0}(\sd X-\sd C;\sd\orientation{S}\tensor{1}\sd\sheaf{1}))
    \cong\dicohomology{0}(\sd X-\sd C;\orientation{S}\tensor{1}\sd\sheaf{1})
  \end{equation*}
  the first from Proposition \ref{prop:sd-cohomologie} and the second from Lemma \ref{lem:sd-orientation}.  
  The result then follows by Theorem \ref{thm:pd}.
\end{proof}

\begin{proof}[proof of Theorem \ref{thm:equalizer}]
  The natural isomorphism
  $$\dihomology{1}(X;\sheaf{1})\cong\dicohomology{0}(X;\orientation{S}\tensor{1}\sheaf{1})$$
  [Theorem \ref{thm:pd}] together with Lemma \ref{lem:pre.pd} gives the result.
\end{proof}

\section{Networks}\label{sec:mfmc}
The sheaf invariants of the previous section describe natural structures on networks.

\subsection{Constraints}\label{subsec:constraints}
Constraints on network dynamics are often local in nature.

\begin{eg}[Numerical]
  \label{eg:numerical.constraints}
  The partial $\N$-subsemimodule
  $$\{0,1,\ldots,\omega_e\}\subset\N$$
  naturally describes all possible quantities of cars on the road $e$ of a network described by an $\N$-weighted digraph $(\digraph{1};\omega)$.
\end{eg}

\begin{eg}[Multicommodities]
  \label{eg:multidimensional.constraints}
  The partial $\N$-subsemimodule
  $$\{v\in\R^2\;|\;v\cdot c\leqslant\omega_e\}\subset\R^{\geqslant 0}\times\R^{\geqslant 0}$$
  describes all possible ratios of two commodities in a supply chain described by an $\R^{\geqslant 0}$-weighted digraph $(\digraph{1};\omega)$ and vector $c\in\R^2$.
\end{eg}

Constraints of interest in information processing [\cite{ghrist2011applications}, Example \ref{eg:circuits}] typically exhibit more interesting restriction maps between cells than mere partial injections.

\begin{eg}[Information Processing]
  \label{eg:circuits}
  Let $\Lambda$ be the Boolean semiring [Example \ref{eg:flatness}].
  Free $\Lambda$-semimodules encode the possible values of bit-strings and $\Lambda$-maps encode logical operations on bit-strings.
  Hence a partial $S$-subsheaf of an $\Lambda$-sheaf encodes the local functionality of a microprocessor with logical processors at the nodes and local channel bandwidths determined by the size of generating sets for the edge stalks.
\end{eg}

Networks also come equipped with distinguished \textit{sources} and \textit{sinks}.
For convenience, this paper follows \cite{frieze1984algebraic,ghrist2011applications} in formally adjoining an edge $e$ to a digraph $X$ with a distinguished ordered pair $(s,t)$ of vertices such that $(\delplus{.4}e,\delminus{.4}e)=(s,t)$; thus a single edge $e$ in a digraph encodes both the source and sink.

\subsection{Cuts}\label{subsec:cuts}
The \textit{value} of a subset $C\subset E_X$ in a weighted digraph $(X;\omega)$ is
$$\sum_{c\in C}\omega_c.$$

\begin{defn}
  For each partial $S$-sheaf $\sheaf{1}$ on a digraph $X$, let
  $$\cutvalue{C}{\sheaf{1}}=\im\;\dicohomology{1}(C\subset X-e)\circ\delta_-:\dicohomology{0}(C;\sheaf{1})\ra\dicohomology{1}(X-e;\sheaf{1}).$$
\end{defn}

\begin{prop}[Cut values]
  \label{prop:cut-values}
  For $C$ a set of edges in a weighted digraph $(X;\omega)$,
  $$\cutvalue{C}{\orientation{S}\tensor{1}\omega}=\left(\sum_{c\in C}\omega_c\right),$$
  where $(b)$ denotes the down-set in the $S$-semimodule $M$ in which $(X;\omega)$ is weighted, with respect to the natural preorder on $M$.
\end{prop}

For a pair $A,B\subset V_X$ of vertices in a digraph $X$, $A:B$ denotes the set
$$A:B=(\delminus{.4}^{-1}A\cap\delplus{.4}^{-1}B)\subset E_X.$$

An \textit{$(s,t)$-cut} is a partition $A,V_X-A$ of the vertices $V_X$ such that $s\in A$ and $t\notin A$, for each choice $s,t\in V_X$ of vertices in a given digraph $X$.
Following our convention for encoding a distingushed source and sink node in a network as a single edge, an \textit{$e$-cut} will refer to a subset of $E_X$ of the form $A:V_X-A$ for $A,V_X$ an $(\delplus{.4}e,\delminus{.4}e)$-cut of a digraph $X$, for each choice of an edge $e\in E_X$ in a given digraph $X$.

\begin{lem}
  \label{lem:cuts}
  Fix $e\in E_X$ with $X-e$ acyclic.
  The following are equivalent forfinite  $C\subset X$.
  \begin{enumerate}
    \item $C$ is an $e$-cut
    \item The element in $H^1(X;\constantsheaf{S})$ represented by $e\in S[e]$ is represented by $\sum_{c\in C}c$ in $S[C]$.
  \end{enumerate}
\end{lem}

A proof is given by induction on the length of maximal chain of $X-e$, regarded as a poset whose partial order $\leqslant_{X-e}$ is generated by relations of the form $\delminus{.4}e\leqslant_{X-e}\delplus{.4}e$.

\begin{eg}
  \label{eg:cuts}
  In each of the graphs, the dashed edges define $e$-cuts.
  
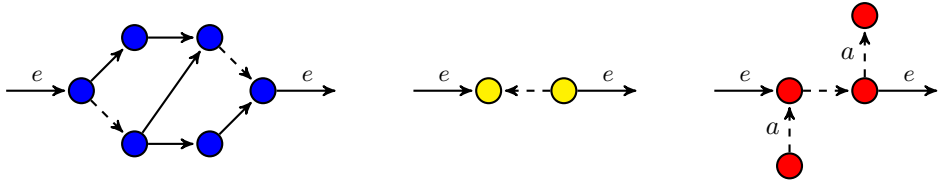
\vspace{.1in}
  \begin{center}
  \begin{tikzpicture}[->,>=stealth',shorten >=1pt,auto,node distance=1cm,
                    thick,main node/.style={circle,fill=red!10,draw,font=\Large\bfseries}]
  \coordinate (0);
  \node[main node,fill=blue] (1) [right of =0] {};
  \node[main node,fill=blue] (2) [above right of =1] {};
  \node[main node,fill=blue] (3) [right of =2] {};
  \node[main node,fill=blue] (4) [below right of =3] {};
  \node[main node,fill=blue] (5) [below right of =1] {};
  \node[main node,fill=blue] (6) [right of =5] {};
  \coordinate [right of =4] (7);
 \path[every node/.style={font=\small}]
    (0) edge node [sloped, pos=.5] {$e$} (1)
    (1) edge node [right] {} (2)
    (2) edge node [right] {} (3)
    (5) edge node [right] {} (6)
    (5) edge node [right] {} (3)
    (6) edge node [right] {} (4)
    (4) edge node [sloped, pos=.5]{$e$} (7);
  \path[dashed]
    (3) edge node [right] {} (4)
    (1) edge node [right] {} (5);
  \coordinate [right of =7] (8);
  \node[main node,fill=yellow] (9) [right of =8] {};
  \node[main node,fill=yellow] (10) [right of =9] {};
  \coordinate [right of =10] (11);
  \path[every node/.style={font=\small}]
    (8) edge node [sloped, pos=.5] {$e$} (9)
    (10) edge node [sloped, pos=.5] {$e$} (11);
  \path[dashed]
    (10) edge node [right] {} (9);
  \coordinate [right of =11] (12);
  \node[main node,fill=red] (13) [right of =12] {};
  \node[main node,fill=red] (14) [below of =13] {};
  \node[main node,fill=red] (15) [right of =13] {};
  \node[main node,fill=red] (17) [above of =15] {};
  \coordinate [right of =15] (16);
 \path[every node/.style={font=\small}]
    (12) edge node [sloped, pos=.5] {$e$} (13)
    (15) edge node [sloped, pos=.5] {$e$} (16);
  \path[dashed]
    (13) edge node [] {} (15)
    (15) edge node [] {$a$} (17)
    (14) edge node [] {$a$} (13);
  \end{tikzpicture} 
  \end{center} 
\end{eg}

\subsection{Flows}\label{subsec:flows}
An \textit{$\omega$-flow} on a digraph $(X;\omega)$ is a function
$$\phi:\edges{\digraph{1}}\ra M$$
to the commutative monoid $M$ of edge weights satisfying the following pair of conditions:
\begin{enumerate}
  \item\label{item:capacity.constraint} $\phi(e)\leqslant_{M}\omega_e$ for each $e\in E_X$
  \item\label{item:conservation.law} $\sum_{e\in\delminus{.4}^{-1}(v)}\phi(e)=\sum_{e\in\delplus{.4}^{-1}(v)}\phi(e)$ for each $v\in V_X$.
\end{enumerate}

Such flows on weighted digraphs generalize to sheaf-valued flows in the following sense.
Condition (\ref{item:capacity.constraint}) generalizes to the structure of a partial $S$-sheaf.
Condition (\ref{item:conservation.law}) generalizes to an equalizer diagram.

\begin{defn}
  An \textit{$\sheaf{1}$-flow} is an element in the equalizer of the diagram
  \begin{equation}
  \label{eqn:flow-defn}
  \xymatrix@C=2pc{
    \prod_{e\in\edges{\digraph{1}}}\dicohomology{0}(\langle e\rangle;\sheaf{1})
      \ar@<.7ex>[rrrr]^-{\dicohomology{0}(\delminus{.4} e\subset\langle e\rangle)}
      \ar@<-.7ex>[rrrr]_-{\dicohomology{0}(\delplus{.4}e\subset\langle e\rangle)}
  & & & & \pseudoprod_{v\in\vertices{\digraph{1}}}\sheaf{1}(v),
  }
  \end{equation}
  where $\prod$ denotes the Cartesian product of underlying sets equipped with coordinate-wise operations, for each partial $S$-sheaf $\sheaf{1}$ on a locally finite digraph $X$.
  The \textit{support}, written $|\phi|$, of an $\sheaf{1}$-flow $\phi$ is the union of $\langle e\rangle$ for all $e\in E_X$ with the $e$-indexed of $\phi$ in $\dicohomology{0}(\langle e\rangle;\sheaf{1})$ non-zero.
  An $\sheaf{1}$-flow $\phi$ is \textit{finite} if $|\phi|$ is finite and \textit{$e$-acyclic} if $|\phi|-e$ is acyclic.
  An $\sheaf{1}$-flow is \textit{locally $S$-decomposable} if it lifts to an $\sheaf{1}_0$-flow for $\sheaf{1}_0\ra\sheaf{1}$ a natural transformation from a flat partial $S$-sheaf $\sheaf{1}_0$.
\end{defn}


\begin{prop}
  \label{prop:flows}
  For each partial $S$-sheaf $\sheaf{1}$ on a digraph $X$, 
  \begin{equation}
    \label{eqn:flows}
    \dihomology{1}(X;\sheaf{1})=\substack{\text{finite and}\\\text{locally decomposable}\\\text{$\sheaf{1}$-flows}}
  \end{equation}
  naturally.
  The above partial $S$-semimodule contains all finite $\sheaf{1}$-flows if $S$ is a ring or for each $v\in V_X$, the in-degree of $v$ is $1$, the out-degree of $v$ is $1$, or $\sheaf{1}(v)$ is flat.
\end{prop}
\begin{proof}
  The second statement follows from Theorem \ref{thm:equalizer} because finite direct sums are Cartesian products of underlying sets equipped with coordinate-wise operations.

  Hence the left side of (\ref{eqn:flows}) is naturally the image of $\sheaf{1}_0$-flows under a natural partial $S$-homomorphism from $\sheaf{1}_0$-flows to $\sheaf{1}$-flows induced by a c-sectionwise surjection $\sheaf{1}_0\ra\sheaf{1}$ with $\sheaf{1}_0$ objectwise projective and hence flat.
\end{proof}

\begin{eg}[Indecomposability]
  \label{eg:bifurcations}
  Given the sup-semilattice $\Lambda$ having Hasse diagram in Example \ref{eg:flatness} and digraph $X$ given below, the $\constantsheaf{\Lambda}$-flow illustrated below is not locally $\N$-decomposable and does not lie in $\dihomology{1}(X;\constantsheaf{\Lambda})$.
  
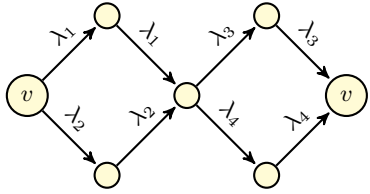
\label{fig:bifurcations}
  \begin{center}
  \begin{tikzpicture}[->,>=stealth',shorten >=1pt,auto,node distance=1.5cm,
                    thick,main node/.style={circle,fill=yellow!20,draw,font=\sffamily\small}]

  \node[main node] (1) {};
  \node[main node] (2) [below right of=1] {};
  \node[main node] (3) [below left of=2] {};
  \node[main node] (4) [above right of=2] {};
  \node[main node] (5) [below right of=2] {};
  \node[main node] (6) [below left of=1] {$v$};
  \node[main node] (7) [below right of=4] {$v$};

  \path[every node/.style={font=\small}]
    (1) edge node [sloped, pos=.1] {$\lambda_1$} (2)
    (3) edge node [sloped, pos=.9] {$\lambda_2$} (2)
    (2) edge node [sloped, pos=.9] {$\lambda_3$} (4)
    (2) edge node [sloped, pos=.1] {$\lambda_4$} (5)
    (6) edge node [sloped, pos=.9] {$\lambda_1$} (1)
    (6) edge node [sloped, pos=.1] {$\lambda_2$} (3)
    (4) edge node [sloped, pos=.1] {$\lambda_3$} (7)
    (5) edge node [sloped, pos=.9] {$\lambda_4$} (7);
  \end{tikzpicture} 
  \end{center} 
\end{eg}


\subsection{MFMC}\label{subsec:mfmc}
Henceforth fix an edge $e\in E_X$.
A decomposition of the values of $\sheaf{1}$-flows in terms of $\sheaf{1}$-values of cuts generalizes MFMC.
This decomposition is really a \textit{homotopy limit}.
In order to make the analogy with classical homotopy theory clear, the homotopy limit defined by the proposition below can be interpreted as a right Kan extension to a suitable localization of generalized resolutions (kernel-pairs) of partial $S$-sheaves.

\begin{prop}
  \label{prop:holim}
  There exists a terminal natural transformation from a functor
  \begin{equation}
    \label{eqn:holim}
      \holim_{C}\cutvalue{C}{-}:\Sh_{X;\CONSTRAINTS}\ra\CONSTRAINTS
  \end{equation}
  to $\bigcap_C\cutvalue{C}{-}$, where $C$ denotes an element in a given collection of subsets of $X$, terminal among all such natural transformations from functors sending c-sectionwise surjections to surjections.
  For $\sheaf{1}$ a direct sum in $\Sh_{X;\CONSTRAINTS}$ of pushforwards of constant sheaves, $\holim_{C}\cutvalue{C}{\sheaf{1}}=\bigcap_C\cutvalue{C}{\sheaf{1}}$.
\end{prop}
\begin{proof}
  Let $\sheaf{1}$ denote a partial $S$-sheaf and $\mathcal{P}$ denote a direct sum in $\Sh_{X;\CONSTRAINTS}$ of pushforwards of constant sheaves at projective partial $S$-semimodules. 
 
  Let $I$ be the functor $\Sh_{X;\CONSTRAINTS}\ra\CONSTRAINTS$ naturally defined on objects by
  $$I\sheaf{1}=\bigcap_C\cutvalue{C}{\sheaf{1}}.$$
 
  The image $L\sheaf{1}$ of $I\sheaf{1}_0$ under $\dicohomology{1}(X;\eta)$ is independent of a choice $\eta$ of c-sectionwise surjections of the form $\mathcal{P}\ra\sheaf{1}$; for two such c-sectionwise surjections $\epsilon':\mathcal{P}'\ra\sheaf{1}$ and $\epsilon'':\mathcal{P}''\ra\sheaf{1}$, there exists a natural transformation $\eta:\mathcal{P}'\ra\mathcal{P}''$ with $\epsilon''\circ\eta=\epsilon'$, hence $\dicohomology{1}(X;\epsilon')(I\mathcal{P}')\subset\dicohomology{1}(X;\epsilon'')(I\mathcal{P}'')$, and hence also $\dicohomology{1}(X;\epsilon'')(I\mathcal{P}'')\subset\dicohomology{1}(X;\epsilon')(I\mathcal{P}')$ by symmetry.
  Hence $L$ extends a functor 
  $$L:\Sh_{X;\CONSTRAINTS}\ra\CONSTRAINTS$$
  defined on morphisms $\eta$ as restrictions and corestrictions of $I\eta$, and hence natural transformation $\iota:L\ra I$ defined component-wise by inclusion.
  
  For each flat $\semiring{1}$-semimodule $M$, $I\constantsheaf{M}\cong I\constantsheaf{\semiring{1}}\tensor{1}M$ naturally in $M$ [Lemma \ref{lem:HH0-universal.coefficients}], and hence $L\constantsheaf{M}=I\constantsheaf{M}$ because $-\tensor{1}M$ preserves epis [Proposition \ref{prop:tensors}].
  Hence $L\mathcal{P}\cong I\mathcal{P}$ naturally in $\mathcal{P}$.

  For each partial $S$-sheaf $\sheaf{1}$, there exists a c-sectionwise surjection of the form $\mathcal{P}\ra\sheaf{1}$ [Lemma \ref{lem:cogodement}].
  Hence $\iota$ is terminal among all natural transformations to $I$ from a functor sending c-sectionwise surjections to surjections.
\end{proof}

The following theorem expresses feasible flow-values as a homotopy limit of cut-values.

\begin{thm}[Sheaf-Theoretic MFMC]
  \label{thm:mfmc}
  The equality
  \begin{equation}
    \label{eqn:mfmc}
    \flowvalues{e}{X}{\sheaf{1}}=\holim\;\!\!_C\cutvalue{C}{\orientation{S}\tensor{1}\sheaf{1}}\subset\bigcap_C\cutvalue{C}{\orientation{S}\tensor{1}\sheaf{1}},
  \end{equation}
  where the homotopy limit is taken over all minimal $e$-cuts $C$ and the inclusion is an equality for the case $\dihomology{1}(-;\sheaf{1})$ exact at $e$, holds for the following data.
  \begin{enumerate}
    \item cellular sheaf $\sheaf{1}$ of $\semiring{1}$-semimodules on $X$
    \item edge $e$ in $X$ with $X-e$ acyclic
  \end{enumerate}
\end{thm}
\begin{proof}
  Let $C$ denote an $e$-cut.
  Let $[C]$ denote $\cutvalue{C}{\orientation{S}\tensor{1}\sheaf{1}}$ for each $C$.

  There exists a natural transformation $\sheaf{1}_0\ra\sheaf{1}$ from a direct sum $\sheaf{1}_0$ in $\Sh_{X;\CONSTRAINTS}$ of pushforwards of constant sheaves and dotted isomorphism making the diagram
  \begin{equation*}
    \xymatrix{
         \bigcap_{C}\cutvalue{C}{\sheaf{1}_0}
          \ar@{.>}[r]^{\cong}
          \ar[d]
      &  \dihomology{1}(X;\sheaf{1})
          \ar[d]^-{\dihomology{0}(X-e\subset X)\circ\delta_-\circ\dicohomology{0}(e\subset X)}
        \\
         \cutvalue{e}{\orientation{S}\tensor{1}\sheaf{1}}
          \ar[r]
      &  \dicohomology{1}(X;\orientation{S}\tensor{1}\sheaf{1})=\dihomology{0}(X;\sheaf{1}),
    }
  \end{equation*}
  where the left vertical arrow is the composite of inclusion into $\cutvalue{e}{\sheaf{1}_0}$ followed by $\cutvalue{e}{\epsilon}$ and bottom horizontal arrow is inclusion, commute [Lemma \ref{lem:cogodement}].
  The left and right vertical arrows have respective images the left and middle sides of (\ref{eqn:mfmc}) [Proposition \ref{prop:flows}].
  The inclusion in (\ref{eqn:mfmc}) follows by terminality of the natural transformation $\holim_C\cutvalue{C}{-}\ra\bigcap_C\cutvalue{C}{-}$.

  Consider the case $\dihomology{1}(-;\sheaf{1})$ exact at $e$.  
  Take $\lambda\in[e]-\im\dicohomology{0}(e\subset X;\orientation{S}\tensor{1}\sheaf{1})$.
  Then $\delta_-\lambda\neq\delta_+\lambda$ by exactness.
  Then there exists an $e$-cut $C$ such that $\delta_-\lambda\notin[C]$ and $\delta_+\lambda\in[C]$ [Lemma \ref{lem:cuts}].
  Otherwise by naturality $\delta_-\lambda=\delta_+\lambda$.
\end{proof}

\begin{eg}[Duality Gap]
  \label{eg:gap}
  Consider the digraph $(X;\omega)$
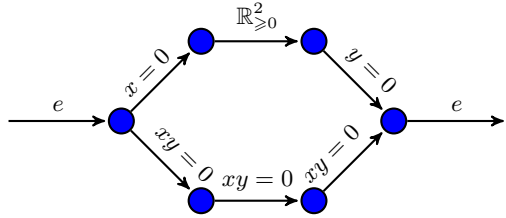
  \begin{center}
  \begin{tikzpicture}[->,>=stealth',shorten >=1pt,auto,node distance=1.5cm,
                    thick,main node/.style={circle,fill=yellow!20,draw,font=\sffamily\small}]
 \coordinate (0);
  \node[main node,fill=blue] (1) [right of =0] {};
  \node[main node,fill=blue] (2) [above right of =1] {};
  \node[main node,fill=blue] (3) [right of =2] {};
  \node[main node,fill=blue] (4) [below right of =3] {};
  \node[main node,fill=blue] (5) [below right of =1] {};
  \node[main node,fill=blue] (6) [right of =5] {};
  \coordinate [right of =4] (7);
 \path[every node/.style={font=\small}]
    (0) edge node [sloped, pos=.5] {$e$} (1)
    (1) edge node [sloped, pos=.9] {$x=0$} (2)
    (2) edge node [sloped, pos=.5] {$\R_{\geqslant 0}^2$} (3)
    (5) edge node [sloped, pos=.5] {$xy=0$} (6)
    (6) edge node [sloped, pos=.9] {$xy=0$} (4)
    (4) edge node [sloped, pos=.5]{$e$} (7)
    (3) edge node [sloped, pos=.1] {$y=0$} (4)
    (1) edge node [sloped, pos=.1] {$xy=0$} (5);
  \end{tikzpicture} 
  \vspace{.1in}
  \end{center} 
  weighted in $\R_{\geqslant 0}^2$, where $\omega_e=\R_{\geqslant 0}^2$.
  The feasible $e$-values of (finite, decomposable) $\omega$-flows form the partial $\N$-subsemimodule of $\R^2_{\geqslant 0}$ defined by $xy=0$.  
  However, $\cutvalue{C}{\orientation{S}\tensor{1}\sheaf{1}}=\R^2_{\geqslant 0}$ for each $e$-cut $C$.
  For $(1,1)\in\omega(e)$, $\delta_-(1,1)=\delta_+(1,1)$ are classes in $\dicohomology{1}(X-e;\orientation{S}\tensor{1}\sheaf{1})$ both represented by elements in $\directsum_{c\in C}\omega(c)$ for each $e$-cut $C$.
\end{eg}

\begin{cor}[Algebraic MFMC]
  \label{cor:mfmc}
  The equality
  $$\bigvee_{\substack{\text{finite and}\\\text{$S$-decomposable}\\\text{flow $\phi$}}}\backspace\backspace\phi(e)=\backspace\;\;\bigwedge_{\substack{\text{$e$-cut $C$}}}\sum_{c\in C}\omega_c.$$
  holds for the following data.
  \begin{enumerate}
    \item naturally complete inf-semilattice ordered $S$-semimodule $M$
    \item digraph $(X;\edgeweights)$ with edges weighted by elements in $M$
    \item edge $e$ in $X$ with $X-e$ acyclic
  \end{enumerate}
\end{cor}
\begin{proof}
  Letting $C$ denote an $e$-cut and $\phi$ denote a flow on $(X;\omega)$,
  $$\left(\bigwedge_{C}\sum_{e\in C}\omega_e\right)=\bigcap_{C}\left(\sum_{e\in C}\omega_e\right)=\bigcap_{C}\cutvalue{C}{\orientation{S}\tensor{1}\omega}=\flowvalues{e}{X}{\omega}=\left(\bigvee_\phi\phi(e)\right),$$
  where $(b)$ denotes the partial $S$-subsemimodule $\{a\;|\;a\leqslant_Mb\}$ for $\leqslant_M$ the natural preorder on $M$, with the first equality following from $M$ naturally inf-semilattice ordered, the second equality following from Proposition \ref{prop:cut-values}, the third equality following from Theorem \ref{thm:mfmc} and Lemma \ref{lem:exactness}, and the last equality following from the naturality of the isomorphism in Proposition \ref{prop:flows} and $M$ naturally complete.
  The conclusion follows because the natural preorder on $M$ is the preorder of a semilattice and hence antisymmetric.
\end{proof}

\addtocontents{toc}{\protect\setcounter{tocdepth}{0}}
\section{Acknowledgements}
The author is indebted to Rob Ghrist both for initially sketching the connection between Poincar\'{e} Duality and optimization and for suggesting countless improvements in exposition and corrections in content, Justin Curry for helpful conversations on cellular sheaves, and both Greg Henselmen and Shiying Dong for helpful comments and corrections on previous drafts.
This work was supported in part by federal contracts FA9550-12-1-0416, FA9550-09-1-0643, and HQ0034-12-C-0027.

\bibliography{gv}{}
\bibliographystyle{plain}
\end{document}